\documentclass[10pt,a4paper,oneside]{article}
\usepackage[english]{babel}
\usepackage{a4wide,amsmath,amsthm,amssymb,url,graphicx,xspace,algorithm,algorithmic,pgf}
\usepackage[latin1]{inputenc}
\usepackage{enumitem}
\usepackage{multirow}
\usepackage{hyperref}
\usepackage{tabularx}
\usepackage{tikz}
\usepackage{cases}
\usepackage{allrunes}
\usepackage{bm}

\def\F{\mathbb{F}}
\def\Z{\mathbb{Z}}

\def\A{\mathcal{A}}
\def\S{\mathcal{S}}

\DeclareMathOperator{\PG}{PG}
\DeclareMathOperator{\GL}{GL}
\DeclareMathOperator{\PGL}{PGL}
\DeclareMathOperator{\PGammaL}{P{\Gamma}L}
\DeclareMathOperator{\AG}{AG}
\DeclareMathOperator{\Aut}{Aut}

\DeclareMathOperator{\Tr}{Tr}

\newtheorem{theorem}{Theorem}[section]

\newtheorem{lemma}[theorem]{Lemma}

\newtheorem{corollary}[theorem]{Corollary}

\theoremstyle{definition}
\newtheorem{definition}[theorem]{Definition}
\newtheorem{remark}[theorem]{Remark}
\newtheorem{construction}{Construction}
\newtheoremstyle{dotless}{}{}{\itshape}{}{\bfseries}{}{ }{}

  \theoremstyle{dotless}

\newcommand{\comments}[1]{}

\author{ Maarten De Boeck \thanks{This author is supported by the BOF (Special Research Fund) of Ghent University}  \and Geertrui Van de Voorde \thanks{This author is a postdoctoral fellow of the Research Foundation -- Flanders (FWO).}}
\title{Elation KM-arcs}
\date{}
\begin{document}
\maketitle

\begin{abstract}
In this paper, we study {\em KM-arcs} in $\PG(2,q)$, the Desarguesian projective plane of order $q$. A KM-arc $\mathcal{A}$ of type $t$ is a natural generalisation of a {\em hyperoval}: it is a set of $q+t$ points in $\PG(2,q)$ such that every line of $\PG(2,q)$ meets $\mathcal{A}$ in $0,2$ or $t$ points.

We study a particular class of KM-arcs, namely, {\em elation} KM-arcs. These KM-arcs are highly symmetrical and moreover,  many of the known examples are elation KM-arcs. We provide an algebraic framework and show that all elation KM-arcs of type $q/4$ in $\PG(2,q)$ are translation KM-arcs. Using a result of \cite{wij}, this concludes the classification problem for elation KM-arcs of type $q/4$. 

Furthermore, we construct for all $q=2^h$, $h> 3$, an infinite family of elation KM-arcs of type $q/8$, and for $q=2^h$, where $4,6,7\mid h$ an infinite family of KM-arcs of type $q/16$. Both families contain new examples of KM-arcs.
\end{abstract}

 \paragraph*{Keywords:} KM-arc, $(0,2,t)$-arc, set of even type, elation arc, translation arc
 \paragraph*{MSC 2010 codes:} 51E20, 51E21

\section{Introduction and definitions}

%\subsection{KM-arcs}
Point sets in $\PG(2,q)$, the Desarguesian projective plane of over the finite field $\F_q$ of order $q$, that have few different intersections sizes with lines have been a research subject throughout the last decades. A point set $\S$ of \emph{type} $(i_{1},\dots,i_{m})$ in $\PG(2,q)$ is a point set such that for every line in $\PG(2,q)$ the intersection size $\ell\cap\S$ equals $i_{j}$ for some $j$ and such that each value $i_{j}$ occurs as intersection size for some line. In \cite{mig} point sets of type $(0,2,q/2)$ of size $\frac{3q}{2}$ were studied. This led to the following generalisation by Korchmáros and Mazzocca in \cite{km}. 

\begin{definition}
  A \emph{KM-arc of type $t$} in $\PG(2,q)$ is a point set of type $(0,2,t)$ with size $q+t$. A line containing $i$ of its points is called an \emph{$i$-secant}.
\end{definition}

Originally these KM-arcs were denoted as \emph{$(q+t)$-arcs of type $(0,2,t)$} \cite{km} or \emph{$(q+t,t)$-arcs of type $(0,2,t)$} \cite{gw} but in honour of Korchmáros and Mazzocca, the notation `KM-arcs' was introduced in \cite{vdd}. KM-arcs of type $t=2$ are {\em hyperovals}, which have their own theory; the classification of hyperovals seems far out of reach at this moment. KM-arcs of type $t=q$ in $\PG(2,q)$ on the other hand, are easily seen to be the symmetric difference of two lines. For KM-arcs of type $2<t<q$, further combinatorial information and conditions on $t$ and $q$ can be deduced. The following results were obtained in \cite[Theorem 2.5]{gw} and \cite[Proposition 2.1]{km}. 

\begin{theorem}
  If $\mathcal{A}$ is a KM-arc of type $t$ in $\PG(2,q)$, $2<t<q$, then
  \begin{itemize}
    \item $q$ is even;
    \item $t$ is a divisor of $q$;
    \item there are $q/t+1$ different $t$-secants to $\mathcal{A}$, and they are concurrent.
  \end{itemize}
\end{theorem}

If $\A$ is a KM-arc of type $t>2$, then the point contained in all $t$-secants to $\A$ is called the \emph{$t$-nucleus} of $\A$.

The main questions in the study of the KM-arcs are for which values of $q$ and $t$, a KM-arc of type $t$ in $\PG(2,q)$ exists, and which nonequivalent KM-arcs of type $t$ in $\PG(2, q)$ exist for given admissible $q$ and $t$.  

Recall that every element of $\PGammaL(3,q)$ defines a {\em collineation} of the projective plane $\PG(2,q)$ and vice versa, where a collineation is an incidence preserving mapping. KM-arcs are studied up to $\PGammaL$-equivalence.  {\em Elations} of a projective plane are particular collineations that will play an important role in this paper. An {\em elation} with axis the line $\ell$ and centre the point $R$ on $\ell$ is a collineation which fixes the points of $\ell$ and stabilises the lines through the centre $R$. We see that the set of all elations with a fixed centre and a fixed axis form a subgroup of $\PGammaL$.

A KM-arc is a {\em translation} KM-arc with translation line $\ell_\infty$ if the group of all elations with axis $\ell_\infty$ that stabilise $\A$, acts transitively on the points of $\A\setminus \ell_\infty$ (see \cite{km}).

\begin{definition}
	Let $\A$ be a KM-arc of type $t>2$ in $\PG(2,q)$ with $t$-nucleus $N$. Then $\A$ is an {\em elation KM-arc} with elation line $\ell_\infty$ if and only if for every $t$-secant $\ell\neq \ell_\infty$ to $\mathcal{A}$, the group of elations with axis $\ell_\infty$ that stabilise $\A$ acts transitively on the points of $\ell$.
	\par A hyperoval (KM-arc of type 2) $\mathcal{H}$ in $\PG(2,q)$ is called an \emph{elation hyperoval} with elation line $\ell_\infty$ if a non-trivial elation with axis $\ell_\infty$ which stabilises $\mathcal{H}$ exists.
\end{definition}

It is immediate that all collineations which stabilise a KM-arc of type $t>2$, fix its $t$-nucleus. Hence, the $t$-nucleus of a KM-arc of type $t$ lies on the elation line since all fixed points of an elation lie on the axis. Moreover the $t$-nucleus will be the centre of all elations stabilising the $t$-secants. So, for an elation KM-arc $\mathcal{A}$ of type $t>2$ the group of elations with axis the elation line and centre the $t$-nucleus, stabilising $\mathcal{A}$, acts sharply transitively on the points of $\mathcal{A}\cap \ell$ with $\ell$ an arbitrary $t$-secant.
\par For hyperovals there is no concurrency point of all $2$-secants. Hence, for elation hyperovals, we have presented a slightly modified version of the definition of an elation KM-arc. We will see in Lemmas \ref{Sinvoeren} and \ref{Sinvoeren2} that we can work with both definitions in the same way. Considering that hyperovals are KM-arcs we will call elation hyperovals also elation KM-arcs
\par It follows from the definitions that every translation KM-arc is an elation KM-arc. The following theorem was shown for translation KM-arcs in \cite[Prop. 6.2]{km}. The proof presented there however cannot not be generalised to the case of elation KM-arcs.

\begin{theorem}\label{elationlinesecant}
	Let $\A$ be an elation KM-arc of type $t$ in $\PG(2,q)$, $2\leq t<q$, with elation line $m$, then $m$ is a $t$-secant to $\mathcal{A}$.
\end{theorem}
\begin{proof} Recall that any collineation that stabilises $\A$ has to fix the $t$-nucleus $N$ of $\A$ if $t>2$ and hence, that the elation line $m$ is a line through $N$. If $t=2$, we define $N$ as the centre of the given non-trivial elation that stabilises $\A$.
\par Suppose that the elation line $m$ is not a $t$-secant and let $G$ be the group of elations with centre $N$ and axis $m$ stabilising $\A$. Then $G$ acts transitively on the points of $\A\cap\ell$ for every $t$-secant $\ell$. Note that $G$ has size $t$.
\par Let $\ell_1$ and $\ell_2$ be two $t$-secants of $\A$ trough $N$ (recall that all $t$-secants contain $N$ if $t>2$), and let $P\in \ell_1$ and $Q\in \ell_2$ be two points of $\A$. Denote the intersection $\langle P,Q\rangle\cap m$ by $R$. The orbit of the line $\langle P,Q\rangle$ under $G$ is a set of $t$ lines through $R$. Moreover, every line through $R$ which meets $\A$ in $2$ points, is contained in an orbit of $G$ of length $t$. Hence, the number of $2$-secants through $R$ is a multiple of $t$. Given the fact that $\A$ has $q+t$ points and every line through $R$ meets $\A$ in $0$ or $2$ points, we know that there are $\frac{q+t}{2}$ distinct $2$-secants to $\A$ through $R$. We have that $\frac{q+t}{2}=t\left(\frac{q}{2t}+\frac{1}{2}\right)$, and since $q=2^h$ and $t=2^j$ for some $j<h$, we have that $2t \mid q$. Hence, the  number of $2$-secants through $R$, $(q+t)/2$, is not a multiple of $t$, a contradiction.
\end{proof}

We first introduce the `classical' examples constructed by Korchmáros-Mazzocca and Gács-Weiner and then give a survey of the known results in Table \ref{overviewKMarcs}.

\begin{construction} \cite{km}\label{kmc}
Let $q=2^h$ and $q'=2^{h-i}$, with $h-i \mid h$, and let $L$ be the relative trace function from $\F_{q}$ to $\F_{q'}$. Let $g$ be an $o$-polynomial in $\F_{q'}$, i.e. $(1,g(x),x)$ is the affine part of a hyperoval (KM-arc of type $2$) in $\PG(2,q')$ containing $(0,1,0)$ and $(0,0,1)$. Then, the point set $\{(1,g(L(x)),x)\mid x\in\F_{q}\}$ can uniquely be extended to a KM-arc $\mathcal{A}_{km}$ of type $2^{i}$ in $\PG(2,q)$. It has $2^{i}$-nucleus $(0,0,1)$. 
\end{construction}

We will show in Lemma \ref{alleselation} that all the KM-arcs arising from Construction \ref{kmc} are elation KM-arcs. It was already shown in \cite[Proposition 6.4]{km} that a KM-arc in $\PG(2,q)$ constructed in this way is a translation KM-arc if and only if $g$ is the $o$-polynomial $x\mapsto x^{2^n}$, for an integer $n$ admitting $\gcd(h',n)=1$. Note that $g$ is the $o$-polynomial corresponding to a translation hyperoval in $\PG(2,q')$. 

We now recall the three different constructions from Gács and Weiner \cite{gw}.

\begin{construction} \cite[Construction 3.4]{gw}\label{gwc} Let $I$ be a direct complement of $\F_q$ in the additive group of $\F_{q^h}$, $h >1$. Let $H$ be a hyperoval or a KM-arc of type $t$ with affine part $\{(1,x_k, y_k) : x_k, y_k \in \F_q\} \subseteq \PG(2,q)$. Construct the following point set in $\AG(2, q^h)$:
\[J =\{(1,x_k, y_k+i) \mid (1,x_k, y_k) \in H, i\in I\}\;.\] 
\begin{itemize}
\item[(A)] If $H$ is a hyperoval and $(0,0,1)\in H$, then $J$ can be uniquely extended to a KM-arc of type $q^{h-1}$ in $\PG(2,q^h)$. This KM-arc has $q^{h-1}$-nucleus $(0,0,1)$.
\item[(B)] if $H$ is a hyperoval and $(0,0,1)\notin H$, then $J$ can be uniquely extended to a KM-arc of type $2q^{h-1}$ in $\PG(2, q^h)$. This KM-arc has $2q^{h-1}$-nucleus $(0,0,1)$.
\item[(C)] If $H$ is a KM-arc of type $t$ and $(0,0,1)$ is the $t$-nucleus of $H$, then $J$ can be uniquely extended to a KM-arc of type $tq^{h-1}$ in $\PG(2, q^h)$. This KM-arc has $tq^{h-1}$-nucleus $(0,0,1)$.
\end{itemize}
Note that in construction (A) the hyperoval $H$ contains one more point on $X=0$ in $\PG(2,q)$, next to $(0,0,1)$. Hence, to extend $J$ to a KM-arc of type $q^{h-1}$ we need $q^{h-1}$ points on $X=0$ in $\PG(2,q^{h})$. In constructions (B) and (C) the KM-arc $H$ of type $t$, which can be a hyperoval, can either be completely contained in the affine part of $\PG(2,q)$ or else have $t$ points on $X=0$. In the latter case, to extend $J$ to a KM-arc of type $tq^{h-1}$ we need $tq^{h-1}$ points on $X=0$ in $\PG(2,q^{h})$. In the former case no points need to be added.
%{ Punten op oneindig? Niet altijd van toepassing!}
\end{construction}

\begin{remark} It was already noted by G\'acs and Weiner \cite[p.138]{gw} that Construction \ref{gwc} (A) is equivalent to Construction \ref{kmc}, where the $o$-polynomial $g$ used in Construction \ref{kmc} is the $o$-polynomial of the hyperoval $H$ used in Construction \ref{gwc} (A).% Moreover, as a KM-arc arising from Construction \ref{kmc} is a translation KM-arc if and only if $g$ is the $o$-polynomial of a translation hyperoval, this shows that a KM-arc constructed from a hyperoval $\mathcal{H}$ in Construction \ref{gwc} (A) is a translation KM-arc if and only if the hyperoval $\mathcal{H}$ is a translation hyperoval. As said before, we will see in Lemma \ref{alleselation} that if $\mathcal{H}$ is not a translation hyperoval, the obtained KM-arc is an elation KM-arc.
\end{remark}

\begin{remark}\label{directcomplements}
	If $I$ and $I'$ are different direct complements of $\F_q$ in $\F_{q^h}$, then the affine point set $J=\{(1,x_k, y_k+i) : (1,x_k, y_k) \in H, i\in I\}$ and the affine point set $J'=\{(1,x_k, y_k+i) : (1,x_k, y_k) \in H, i\in I'\}$ in $\PG(2,q)$ are $\PGL$-equivalent. To see this, we will see below that there is an $\F_q$-linear map $\tau$ acting on $\F_{q^h}$ which fixes $\F_q$ but maps $I$ onto $I'$. The map $\phi_\tau:(1,x,y)\mapsto (1,x,\tau(y))$ then induces a collineation of $\PG(2,q^h)$ mapping the point set $J$ onto $J'$, since $\phi_{\tau} (1,x_k,y_k+I)=(1,x_k,\tau(y_k+I))=(1,x_k,y_k+\tau(I))=(1,x_k,y_k+I')$.
	\par To see that there exists an $\F_q$-linear map $\tau$ acting on $\F_{q^h}$ which fixes $\F_q$ but maps $I$ onto $I'$, consider $\F_{q^h}$ as $\F_q^h$, and further consider $\F_q^h$ as a projective space $\PG(h-1,q)$. Then we need to find a collineation fixing the point $V$ corresponding to $\F_q$ and mapping $\pi_{1}$ onto $\pi_{2}$ with $\pi_1$ and $\pi_2$ two different hyperplanes not through $V$. Clearly, there is an elation with axis $\langle V,\pi_{1}\cap\pi_{2} \rangle$ which fulfils the requirements.
\end{remark}

This paper is organised as follows. In Lemma \ref{alleselation}, we will see that the KM-arcs constructed from Construction \ref{gwc} (A) are always elation KM-arcs. We will prove in Lemma \ref{elationals} and Theorem \ref{elationals2} that a KM-arc obtained in Construction \ref{gwc} (B) and (C) is an elation KM-arc if and only if the KM-arc (or hyperoval) started with is an elation KM-arc with $t$-nucleus $(0,0,1)$ (or is an elation hyperoval stabilised by an elation with centre $(0,0,1)$).

\begin{table}[ht]
  \begin{tabular}{c|c|c|c|c}
    $q$ & $t$ & Condition & Comments & Reference \\\hline\hline
    %$2^{h}$ & $2^{i}$ & $h-i\mid h$ & construction based on $\T:\F_{2^{h}}\to\F_{2^{h-i}}$ & \cite[Sect. 5]{km}\\
    %& & & and an $o$-polynomial $g$ over $\F_{2^{h-i}}$, & \\
    %& & & translation iff $g(z)=z^{2^{d}}$ with $(d,h-i)=1$ & \\
    $2^{h}$ & $2^{i}$ & $h-i\mid h$ & elation, see Lemma \ref{alleselation}& \cite{km} (see Constr. \ref{kmc} and \ref{gwc} (A))\\\hline
    $2^{h}$ & $2^{i+1}$ & $h-i\mid h$ & some elation, see Lemma \ref{elationals} & \cite{gw} (see Constr. \ref{gwc} (B))\\\hline
    $2^{h}$ & $2^{i+m}$ & $h-i\mid h$, a KM-arc & some elation, see Lemma \ref{elationals} & \cite{gw} (see Constr. \ref{gwc} (C))\\
     & & of type $2^{m}$ in & & \\
     & & $\PG(2,2^{h-i})$ exists & & \\\hline
    $2^{h}$ & $2^{h-1}$ &  & translation & \cite{wij,vdd}\\
     & &  & (all $\PGL$-equivalent) & \\\hline
    $2^{h}$ & $2^{h-2}$ &$h\geq3$  &  translation \cite{wij} & \cite{vdd}\\
     & & &  non-translation & \cite{wij} \\
     & & & elation iff translation &See Section \ref{secq/4} \\\hline
    $2^{h}$ & $2^{h-3}$ & $h\geq4$ & elation, {\bf new family} & See Section \ref{q/8} \\\hline
    32 & 4 & & one elation, see Remark \ref{qover8in32} & \cite{kmm}, \cite{vdd2} \\\hline
    $2^{h}$ & $2^{h-4}$ & $4\mid h,\ 6 \mid h, \mathrm{or }\ 7 \mid h$ & elation, & See Section \ref{q/16} \\
     & & & {\bf new family for $7\mid h$} & \\\hline
   % $2^{h}$ & $2^{h-5}$ & $5 \mid h$ & elation, {\bf new family??}& See Section \ref{} \\\hline
    %64 & 8 & &  see Section \ref{}, elation/not elation?? & \cite{lim} \\\hline
  \end{tabular}
  \caption{An overview of the known KM-arcs of type $t$ in $\PG(2,q)$}
  \label{overviewKMarcs}
\end{table}

We have seen that every translation KM-arc is an elation KM-arc, but the converse does not necessarily hold. For KM-arcs of type $q/4$ in $\PG(2,q)$ however, we will show in Section \ref{secq/4} that every elation KM-arc is necessarily a translation KM-arc, which brings us to the full classification of elation KM-arcs of type $q/4$ in Theorem \ref{mainq/4}.

The other sections of this paper are devoted to the construction of new examples of KM-arcs; we present a family of elation KM-arcs of type $q/8$ for all values of $q=2^h$ in Section \ref{q/8} and a family of elation KM-arcs of type $q/16$ for $q=2^h$, with $h$ a multiple of $4$, $6$ or $7$, in Section \ref{q/16}.

\section{Elation KM-arcs: an algebraic approach}

In the following lemma, we show that the {\em affine point set} of an elation KM-arc (i.e. the set of points not lying on the elation line) has a convenient algebraic description.
\begin{lemma}\label{Sinvoeren} If $\mathcal{A}$ is an elation KM-arc of type $t>2$ in $\PG(2,q)$, $q=2^h$, with elation line $\ell_\infty: X=0$ and $t$-nucleus $N(0,0,1)$, then there is an additive subgroup $S$ of size $t$ in $\F_{q}$, such that for any $\alpha\in\F_{q}$ the set $\{z\mid(1,\alpha,z)\in\mathcal{A}\}$ is either empty or a coset of $S$. Vice versa, if for a KM-arc $\mathcal{A}$ there is an additive subgroup $S$ of size $t$ in $\F_{q}$, such that for any $\alpha\in\F_{q}$ the set $\{z\mid(1,\alpha,z)\in\mathcal{A}\}$ is either empty or a coset of $S$, then $\mathcal{A}$ is an elation KM-arc with elation line $X=0$ and $t$-nucleus $(0,0,1)$.
\end{lemma}
\begin{proof}Let $\mathcal{A}$ be an elation KM-arc of type $t$ in $\PG(2,q)$, $q=2^h$, with elation line $\ell_\infty: X=0$ and $t$-nucleus $N(0,0,1)$. The $t$-secants, different from $\ell_\infty$, are of the form $Y=\alpha X$ for some $\alpha \in \F_q$. The group $E$ of elations with centre $N$ and axis $\ell_\infty$ consists of the elations induced by all matrices of the form$\left(\begin{smallmatrix} 1&0&0\\0&1&0\\\mu&0&1\end{smallmatrix}\right)$, where $\mu\in\F_{q}$. Here, the points are represented as column vectors, and matrices are acting from the left. It is straightforward to check that a set $T$ is an additive subgroup of $\F_{q}$ if and only if the set $\psi(T)$ of elations corresponding to the matrices in $\left\{\left(\begin{smallmatrix} 1&0&0\\0&1&0\\\nu&0&1\end{smallmatrix}\right)\mid \nu\in T\right\}$ is a subgroup of $E$. The orbit of a point $(1,x,y)$ under the action of $\psi(T)$ is the point set $\{(1,x,\nu+y)\mid \nu\in T\}$.
\par If $\mathcal{A}$ is an elation KM-arc, the orbit of the point $(1,\alpha_1,\beta_1)\in\mathcal{A}$ under the subgroup $E_{\mathcal{A}}$ of $E$ stabilising $\mathcal{A}$ is exactly the set of $t$ points of $\mathcal{A}$ on the $t$-secant $\ell:Y=\alpha_1X$. Let $S$ be the additive subgroup of $\F_{q}$ such that $\psi(S)=E_{\mathcal{A}}$. This implies that the set of points on $\ell\cap\mathcal{A}$ equals $\{(1,\alpha_1,\tau+\beta_1)\mid \tau\in S\}$.%, where $S$ is an additive subgroup of $\F_q$ of size $t$.
 
Vice versa, we assume that $S$ is an additive subgroup of size $t$ in $\F_{q}$, such that for any $\alpha\in\F_{q}$ the set $\{z\mid(1,\alpha,z)\in\mathcal{A}\}$ is either empty or a coset of $S$. Let $G$ be the group of elations induced by the matrices of the form $\left(\begin{smallmatrix} 1&0&0\\0&1&0\\\mu&0&1 \end{smallmatrix}\right)$, where $\mu\in S$. Then $G$ has size $t$ and acts transitively on the set of points $\mathcal{A}\cap\ell$ with $\ell:Y=\alpha X$ a $t$-secant of $\mathcal{A}$, for any $\alpha\in\F_{q}$. This means exactly that $\mathcal{A}$ is  an elation KM-arc with elation line $X=0$ and $t$-nucleus $(0,0,1)$.
\end{proof}

For the hyperoval case we have a similar result.

\begin{lemma}\label{Sinvoeren2}
	If $\mathcal{H}$ is a hyperoval in $\PG(2,q)$, $q=2^h$, that is stabilised by a non-trivial elation with elation line $\ell_\infty: X=0$ and centre $N(0,0,1)$, then there is an additive subgroup $S$ of size $2$ in $\F_{q}$, such that for any $\alpha\in\F_{q}$ the set $\{z\mid(1,\alpha,z)\in\mathcal{H}\}$ is either empty or a coset of $S$. Vice versa, if for a hyperoval $\mathcal{H}$ there is an additive subgroup $S$ of size $2$ in $\F_{q}$, such that for any $\alpha\in\F_{q}$ the set $\{z\mid(1,\alpha,z)\in\mathcal{H}\}$ is either empty or a coset of $S$, then $\mathcal{H}$ is stabilised by a nontrivial elation with elation line $X=0$ and centre $(0,0,1)$.
\end{lemma}
\begin{proof}
	This proof is similar to the proof of Lemma \ref{Sinvoeren}, with the 2-secants through the centre taking the place of the $t$-secants, and $E$ the group consisting of the one non-trivial elation that stabilises $\mathcal{H}$ together with the trivial collineation.
\end{proof}

Now we check whether the known constructions give rise to elation KM-arcs. First we deal with the family of KM-arcs constructed by Korchmáros and Mazzocca

\begin{lemma}\label{alleselation}
All KM-arcs in the family of Korchmáros and Mazzocca (Construction \ref{kmc}, \cite{km}) are elation KM-arcs. 
\end{lemma}
\begin{proof} Recall that the set of affine points of the KM-arc $\mathcal{A}_{km}$ in $\PG(2,q)$ is $\{(1,g(L(x)),x)\mid x\in\F_{q}\}$, where $g$ is an $o$-polynomial in $\F_{q'}$ and $L$ is the relative trace function from $\F_{q}$ to $\F_{q'}$, with $q'=2^{h-i}$, $q=2^h$ and $h-i\mid h$. Define $S=\{x\in \F_q \mid L(x)=0\}$, then $S$ is an additive subgroup of $\F_q$ of size $2^{i}$. We claim that for every $\alpha\in\F_{q}$, the set $T_{\alpha}=\{x\in\F_{q}\mid g(L(x))=\alpha\}$ is either empty or a coset of $S$. First note that $\{g(L(x))\mid x \in \F_{q}\}=\F_{q'}$, so $T_{\alpha}$ is empty if $\alpha\notin\F_{q'}$. If $\alpha\in\F_{q'}$, we can find a $\beta\in\F_{q}$ such that $g(L(\beta))=\alpha$. Then $g(L(\beta+s))=g(L(\beta)+L(s))=g(L(\beta))=\alpha$ for all $s\in S$. Since there are exactly $|S|$ solutions $x$ to the equation $g(L(x))=\alpha$, we know $T_{\alpha}=\beta+S$. This proves the claim, and hence, by Lemma \ref{Sinvoeren}, the statement.
\end{proof}

Now we check the constructions by Gács and Weiner. For additive subgroups $G_1$ and $G_2$ of $\F_{q}$, the additive subgroup generated by subgroups $G_1$ and $G_2$ is denoted by $\langle G_1,G_2\rangle$. If $G_2=\langle x\rangle$, then, by abuse of notation, we also write $\langle G_1,x\rangle$ instead of $\langle G_1,\langle x\rangle\rangle$. Using this convention, we denote the additive subgroup generated by the elements $\alpha_1,\alpha_2,\ldots,\alpha_k\in \F_q$ (or equivalently, the $\F_2$-vector subspace spanned by these elements when considering $\F_{q}$ as a vector space over $\F_{2}$) by $\left\langle\alpha_{1},\alpha_{2},\ldots,\alpha_k\right\rangle$.

\begin{lemma}\label{elationals} Let $\mathcal{B}$ be an elation KM-arc in $\PG(2,q)$ with elation line $X=0$ and $t$-nucleus $(0,0,1)$. Let $\mathcal{A}$ be a KM-arc in $\PG(2,q^h)$ that arises from $\mathcal{B}$ as in Construction \ref{gwc} (C), then $\mathcal{A}$ is an elation KM-arc with elation line $X=0$ and $tq^{h-1}$-nucleus $(0,0,1)$.
Let $\mathcal{H}$ be a hyperoval in $\PG(2,q)$ that is stabilised by a non-trivial elation with elation line $X=0$ and centre $(0,1,0)$ and let $\mathcal{A'}$ be a KM-arc in $\PG(2,q^h)$ that arises from $\mathcal{H}$ as in Construction \ref{gwc} (B), then $\mathcal{A'}$ is an elation KM-arc with elation line $X=0$ and $2q^{h-1}$-nucleus $(0,0,1)$.
\end{lemma}
\begin{proof} 
Since $\mathcal{B}$ is an elation KM-arc, by Lemma \ref{Sinvoeren} we know that there exists an additive subgroup $S$ of size $t$ in $\F_{q}$ such that for any $\alpha\in\F_{q}$ the set $U_{\alpha}=\{z\mid(1,\alpha,z)\in\mathcal{B}\}$ is either empty or a coset of $S$. Now, from the description of $\mathcal{A}$ in Construction \ref{gwc} (C) it follows that $U'_{\alpha}=\{z\mid(1,\alpha,z)\in\mathcal{A}\}$ equals $\{u+i\mid u\in U_{\alpha}, i\in I\}$ if $\alpha\in\F_{q}$, with $I$ a direct complement of $\F_{q}$ in $\F_{q^{h}}$. It is immediate that $U'_{\alpha}$ is empty if $U_{\alpha}$ is empty or if $\alpha\notin\F_{q}$. If $U'_{\alpha}$ is non-empty, then it is a coset of the additive subgroup $\langle S,I\rangle$ of $\F_{q^{h}}$. It now follows from Lemma \ref{Sinvoeren} that $\mathcal{A}$ is an elation KM-arc with elation line $X=0$ and $tq^{h-1}$-nucleus $(0,0,1)$.
\par The proof of the second part is very similar. We now apply Lemma \ref{Sinvoeren2} on $\mathcal{H}$ and then argue as in the first part of the proof.
\end{proof}

We recall that it is not required in Constructions \ref{gwc} (B) and (C) to have non-affine points in the KM-arc to start with. From this point of view it is worthwhile to note that in both cases of the previous lemma non-affine points are required.

We will now prove the converse of Lemma \ref{elationals}.
\begin{theorem}\label{elationals2}
Let $\mathcal{A}$ be a KM-arc of type $tq^{h-1}$, $t>2$, in $\PG(2,q^h)$ with $tq^{h-1}$-nucleus $(0,0,1)$ that arises from some KM-arc $\mathcal{B}$ of type $t$ in $\PG(2,q)$ as in Construction \ref{gwc} (C). If $\mathcal{A}$ is an elation KM-arc with elation line $X=0$, then $\mathcal{B}$ is an elation KM-arc with elation line $X=0$ and $t$-nucleus $(0,0,1)$.% such that $\mathcal{A}$ arises from $\mathcal{B}$ by Construction \ref{gwc} (C).

If $\mathcal{A}$ is an elation KM-arc of type $2q^{h-1}$ in $\PG(2,q^h)$ with $2q^{h-1}$-nucleus $(0,0,1)$ and with elation line $X=0$ that arises from a hyperoval $\mathcal{H}$ as in Construction \ref{gwc} (B), then $\mathcal{H}$ is an elation hyperoval stabilised by a non-trivial elation with axis $X=0$ and centre $(0,0,1)$.
\end{theorem}
\begin{proof} 
We consider the KM-arc $\mathcal{A}$ arising from $\mathcal{B}$ and assume that $\mathcal{A}$ is an elation KM-arc. By Lemma \ref{Sinvoeren} we know that there is an additive subgroup $S'$ of size $tq^{h-1}$ in $\F_{q^{h}}$, such that for any $\alpha\in\F_{q^{h}}$ the set $U'_{\alpha}=\{z\mid(1,\alpha,z)\in\mathcal{A}\}$ is either empty or a coset of $S'$. For $\mathcal{B}$ we define $U_{\alpha}=\{z\mid(1,\alpha,z)\in\mathcal{B}\}$ for any $\alpha\in\F_{q}$. Through Construction \ref{gwc} (C) we know that $U'_{\alpha}$ equals $\bigcup_{u\in U_{\alpha}}u+I$ for any $\alpha\in\F_{q}$, and that it is empty if $\alpha\notin\F_{q}$. We want to prove that there exists an additive subgroup of $\F_{q}$ such that any non-empty $U_{\alpha}$ is a coset of it.
\par It is easy to see that $I\subseteq S'$. Since $I$ is a direct complement of $\F_{q}$ in $\F_{q^{h}}$ we can find an additive subgroup $S$ of $\F_{q}$ such that $S'=\langle S,I\rangle$; necessarily $|S|=t$. Now fix a value $\alpha\in\F_{q}$ such that $U_{\alpha}$ (or equivalently $U'_{\alpha}$) is non-empty. Let $u$ be an element of $U_{\alpha}$. We know that $u$ is also an element of $U'_{\alpha}$, hence we can write $U'_{\alpha}=u+S'$. For an arbitrary $v\in U_{\alpha}$ there exists an $s'\in S'$ such that $v=u+s'$ since $v$ is also an element of $U'_{\alpha}$. Since $S'=\langle S,I\rangle$, there are unique $s\in S$ and $i\in I$ such that $s'=s+i$. So, $v=u+s+i$ and as $v,u,s\in\F_{q}$ and $I$ is a direct complement of $\F_{q}$ we know that $i=0$ and that $v=u+s$. Since $v$ was arbitrarily chosen, we see that $U_{\alpha}\subseteq u+S$. As $|U_{\alpha}|=t=|S|$ we conclude that $U_{\alpha}=u+S$. So, for any $\alpha$ we find that $U_{\alpha}$ is a coset of $S$. The theorem now follows from Lemma \ref{Sinvoeren}.
\par The proof of the second part follows by a very similar reasoning.
\end{proof}

Later in this paper, we will need the notion of $\F_q$-linear sets in a projective space. Let $V$ be an $r$-dimensional vector space over $\F_{q^n}$, let $\Omega$ be the projective space $\PG(V)=\PG(r-1,q^n)$, $q=p^h$, $p$ prime, and let $T$ be a set of points of $\Omega$. The set $T$ is said to be an {\em $\F_q$-linear} set of $\Omega$ of rank $t$ if it is defined by the non-zero vectors of an $\F_q$-vector subspace $U$ of $V$ of dimension $t$, i.e. $T=\mathcal{B}(U)=\{\langle u\rangle_{\F_{q^n}}: u \in U\setminus\{0\}\}$. For more information on $\F_q$-linear sets, we refer to \cite{FQ11} and \cite{olga}.

If $\S$ is an $\F_q$-linear point set contained in a line of $\PG(2,q^{n})$, then the (usual) dual of this point set defines a subset of the set of the lines through a fixed point. We will call such a set {\em an $\F_q$-linear pencil} as the terminology `dual of a linear set' is already in use, see e.g. \cite{olga}.

\begin{remark} Consider an elation KM-arc as in Lemma \ref{Sinvoeren}. The set of points on the $t$-secant $Y=\alpha X$ is of the form $\{(1,\alpha,\beta+s)\mid s\in S\}$ for an additive subgroup $S$ of $\F_{q}$. Now it is clear that this set of points, together with the $t$-nucleus $(0,0,1)$, forms an $\F_2$-linear set on the line $Y=\alpha X$. It has been conjectured by Vandendriessche in \cite{vdd} that {\em all} KM-arcs have this property, i.e., that the points of a KM-arc of type $t$ that lie on a given $t$-secant $\ell$, together with the $t$-nucleus, form an $\F_2$-linear set on $\ell$.
Note that it has been shown in \cite{gw} that the set of points on a $t$-secant $\ell$ to a KM-arc of type $t$ define a {\em Vandermonde set}. A set $\{y_1,\ldots,y_t\}\subseteq \F_q$, with $1<t<q$, is a Vandermonde set if $\sum_i y_i^k=0$ for all $1\leq k \leq t-2$.

Every $\F_2$-linear set is a Vandermonde set, but not all Vandermonde sets are $\F_2$-linear sets. In $\F_{2^h}$, $h\leq 5$, Vandermonde and $\F_2$-linear are equivalent properties, however, in $\F_{2^6}$, where $z^6=z^4+z^3+z+1$ (the default polynomial used in the computer algebra package GAP), the set $\{0,\, z^{12}, z^{15}, z^{17}, z^{19}, z^{43}, z^{56},  z^{59}\}$ is Vandermonde, but not $\F_2$-linear. Considering that all elation KM-arcs have the conjectured property, that the non-elation KM-arcs constructed in \cite[Section 4]{wij} (see Section \ref{secq/4}) have the conjectured property and that in \cite{vdd2} the conjecture was checked for all KM-arcs in $\PG(2,2^h)$, $h\leq 5$, it would be interesting to know whether the Vandermonde property really can be strengthened to the $\F_2$-linear property.
\end{remark}

\section{Elation KM-arcs of type \texorpdfstring{$q/4$}{q/4}}\label{secq/4}

Recall the construction from \cite{wij} where we have permuted the first and third coordinate:
\begin{theorem}\label{ourconstruction}
  Let $\Tr$ be the absolute trace function from $\F_{q}$ to $\F_{2}$, $q=2^{h}\geq8$. Choose $\alpha,\beta\in\F_{q}\setminus\{0,1\}$ such that $\alpha\beta\neq1$ and define
  \[
    \gamma=\frac{\beta+1}{\alpha\beta+1}\;,\quad\xi=\alpha\beta\gamma\;.
  \]
  Now choose $a,b\in\F_{2}\subset\F_{q}$, and define the following sets
  \begin{align*}
    \S_{0}&=\left\{(0,1,z)\mid z\in \F_q,\Tr(z)=0,\Tr\left(\frac{z}{\alpha}\right)=a\right\}\;,\\
    \S_{1}&=\left\{(1,0,z)\mid z\in \F_q,\Tr(z)=0,\Tr\left(\frac{z}{\alpha\gamma}\right)=0\right\}\;,\\
    \S_{2}&=\left\{(1,1,z)\mid z\in \F_q,\Tr(z)=1,\Tr\left(\frac{z}{\alpha\beta}\right)=b\right\}\;,\\
    \S_{3}&=\left\{(1,\gamma,z)\mid z\in \F_q,\Tr\left(\frac{z}{\alpha\gamma}\right)=a+1,\Tr\left(\frac{z}{\xi}\right)=b+1\right\}\;,\\
    \S_{4}&=\left\{(1,\beta+1,z)\mid z\in \F_q,\Tr\left(\frac{z}{\alpha\beta}\right)=a+b+1,\Tr\left(\frac{z}{\xi}\right)=b\right\}\;.
  \end{align*}
  Then, $\A_{\alpha,\beta,a,b}=\cup^{4}_{i=0}\S_{i}$ is a KM-arc of type $q/4$ in $\PG(2,q)$.
\end{theorem}

It is easy to prove (see also \cite[Theorem 4.8]{wij}), that $\A_{\alpha,\beta,a,b}$ is $\PGL$-equivalent to $\A_{\alpha,\beta,0,0}$.
\begin{theorem}\label{transliff}\cite[Theorem 4.9]{wij}
  Let $\alpha,\beta\in\F_{q}\setminus\{0,1\}$, with $\alpha\beta\neq1$. The KM-arc $\A_{\alpha,\beta,0,0}$ is a translation KM-arc if and only if $\alpha\in\left\{\frac{1}{\beta^{2}},1+\frac{1}{\beta},\beta,\frac{1}{\sqrt{\beta}},\frac{1}{\beta+1}\right\}$.
\end{theorem}

Consider an $\F_2$-linear set $\mathcal{S}$ of size $5$ in $\PG(1,2^h)$. By definition, we know that there is an $\F_{2}$-subspace $U$ of $\F^{2}_{2^h}$ such that $\mathcal{S}=\mathcal{B}(U)$. Note that $U$ is not uniquely determined by $\mathcal{S}$. It is not too hard to check that either $\dim(\langle U\rangle_2)=2$ or $\dim(\langle U\rangle_2)=3$, where $\langle U \rangle_2$ denotes the projective space defined by the vector space $U$ over $\F_{2}$. If $\langle U \rangle_2$ is a solid, then every point of $\mathcal{B}(U)$ is defined by the projective points of one line of $\langle U\rangle_2$. If $\langle U \rangle_2$ is a plane, then there is exactly one point $H$ of $\mathcal{B}(U)$ such that $H$ is determined by the points of a projective line of $\langle U\rangle_2$, and each of the other four points of $\mathcal{B}(U)$ is determined by exactly one of the four remaining points of $\langle U\rangle_2$. In the latter case, $\mathcal{S}$ is called a {\em club} or more specifically, a {\em $2$-club of rank $3$} and the point $H$ is called the {\em head} of the club.
In the former case, $\mathcal{S}=\mathcal{B}(U)$ with $\dim(\langle U\rangle_2)=3$, we see that for every plane $\langle V\rangle_2$ of $\langle U\rangle_2$, also
$\mathcal{S}=\mathcal{B}(V)$, and $\mathcal{S}$ is a club, with the head determined by choice of the plane $\langle V\rangle_2$. We see from this argument that the head is not uniquely determined in this case, and that any point of $\mathcal{S}$ can play the role of the head. It will follow from the proof of the following theorem that in the latter case, the club forms an $\mathbb{F}_4$-subline.

\begin{lemma}\label{vwclub}
	The set $\mathcal{C}=\{(1,0), (0,1), (1,1),(\gamma,1), (\beta+1,1)\}\subseteq\PG(1,2^{h})$ is an $\F_2$-linear set if and only if $\beta\in \{\gamma, \frac{1}{\gamma+1},\sqrt{\gamma+1},1+\frac{1}{\gamma},1+\gamma^2\}$. If two of the values in this set coincide, they all coincide and $\mathcal{C}$ forms an $\mathbb{F}_4$-subline.
\end{lemma}
\begin{proof}
	First suppose that $(1,0)$ is the head of the $\F_2$-linear set. Then $(0,1), (1,1),(\gamma,1), (\beta+1,1)$ are linearly dependent over $\F_2$ if and only if $\beta=\gamma$. If the head is $(0,1)$, then $\{(1,0),(1,1),(\gamma,1),(\beta+1,1)\}=\left\{(1,0),(1,1),(1,\frac{1}{\gamma}),(1,\frac{1}{\beta+1})\right\}$ is an $\F_2$-linear set if an only if $1+\frac{1}{\gamma}+\frac{1}{\beta+1}=0$, equivalently $\beta=\frac{1}{\gamma+1}$.
	If the head is $(1,1)$, we use the collineation induced by the matrix $\left(\begin{smallmatrix}1&0\\1&1\end{smallmatrix}\right)$ to map the head $(1,1)$ onto $(1,0)$, and the other points to the points with coordinates $(0,1),(1,1),(\gamma,\gamma+1)$ and $(\beta+1,\beta)$. It follows that $\mathcal{C}$ is an $\F_2$-linear set with head $(1,0)$ if and only if $1+\frac{\gamma}{\gamma+1}+\frac{\beta+1}{\beta}=0$, equivalently $\beta=1+\frac{1}{\gamma}$. If the head is $(\gamma,1)$, then similarly, we use the matrix $\left(\begin{smallmatrix}1&\gamma\\0&1\end{smallmatrix}\right)$ to map the head to $(0,1)$ and the other points to $(1,0)$, $(\gamma,1)$, $(1+\gamma,1)$ and $ (\beta+\gamma+1,1)$, which forms an $\F_2$-linear set if and only if $\gamma^2=\beta+1$. Finally, if the head is $(\beta+1,1)$, we get that $\beta^2=\gamma+1$ in order for $\mathcal{C}$ to define an $\F_2$-linear set.
	
	If two values of $\{\gamma, \frac{1}{\gamma+1},\sqrt{\gamma+1},1+\frac{1}{\gamma},1+\gamma^2\}$ coincide, then we know from the previous reasoning that the linear set $\mathcal{C}$ has more than one head, and hence, that all points can play the role of the head, which implies that all values have to coincide (this can be deduced by direct calculations as well). We find that $\gamma^3=1$ and $\beta=\gamma$. This implies that $\mathcal{C}=\{(1,x)\mid x\in \mathbb{F}_4\}\cup (0,1)$, and hence, $\mathcal{C}$ defines an $\mathbb{F}_4$-subline.
\end{proof}

\begin{lemma}\label{iffTranslation} The $q/4$-secants to the KM-arc $\mathcal{A}_{\alpha,\beta, a, b}$ in $\PG(2,q)$, $q=2^h$, define an $\F_2$-linear pencil if and only if the KM-arc is a translation arc.
\end{lemma}
\begin{proof}  The $q/4$-secants to a KM-arc of the form $\mathcal{A}_{\alpha,\beta, a, b}$ define the set of points $\mathcal{C}=\{(1,0)$, $(0,1)$, $(1,1)$, $(\gamma,1)$, $(\beta+1,1)\}$ in $\PG(1,q)$. From Lemma \ref{vwclub}, we get that $\mathcal{C}$ is an $\F_2$-linear set if and only if $\beta\in \left\{\gamma, \frac{1}{\gamma+1},\sqrt{\gamma+1},1+\frac{1}{\gamma},1+\gamma^2\right\}$. Plugging in $\gamma=\frac{\beta+1}{\alpha\beta+1}$, yields that this condition is equivalent to $$\beta\in\left\{\frac{\beta+1}{\alpha\beta+1},\frac{\alpha\beta+1}{\alpha\beta+\beta}, \sqrt{\frac{\beta+\alpha\beta}{\alpha\beta+1}},\frac{\alpha\beta+\beta}{\beta+1},\frac{\beta^2+\alpha^{2}\beta^{2}}{\alpha^{2}\beta^{2}+1}\right\}.$$ This in turn is equivalent to
$\alpha\in\{\frac{1}{\beta^{2}},1+\frac{1}{\beta},\frac{1}{\beta+1},\beta,\frac{1}{\sqrt{\beta}}\}$, and hence, we conclude by Theorem \ref{transliff}, that the $q/4$-secants to the KM-arc $\mathcal{A}_{\alpha,\beta, a, b}$ define an $\F_2$-linear pencil if and only if the KM-arc is a translation arc.
\end{proof}

\begin{lemma} \label{vormq/4} Every elation KM-arc of type $q/4$ is $\PGL$-equivalent to a KM-arc whose elation line is given by $X=0$ and whose affine points are given by $\{(1,0,s)\mid s\in S\}\cup\{(1,1,1+s)\mid s\in S\}\cup\{(1,\alpha,\alpha'+s)\mid s\in S\}\cup\{(1,\beta,\beta'+s)\mid s\in S\}$ with $S$ an additive subgroup of size $q/4$ of $\F_q$, and with $\alpha,\alpha',\beta,\beta'\in\F_{q}$.
\end{lemma}
\begin{proof} As $\PGL(3,q)$ acts transitively on the frames (4 points in standard position), we may take the $q/4$-nucleus to be $(0,0,1)$, the elation line to be $X=0$, and the points $(1,0,0)$ and $(1,1,1)$ to be contained in the KM-arc. The statement now follows from Lemma \ref{Sinvoeren}.
\end{proof}

Recall that for additive subgroups $G_1$ and $G_2$ of $\F_{q}$ the additive subgroup generated by subgroups $G_1$ and $G_2$ is denoted by $\langle G_1,G_2\rangle$. Note that an additive subgroup of $\F_{q}$ corresponds to a vector subspace of the $h$-dimensional vector space $\F_2^h$. It is well-known (see e.g \cite[2.24]{lidl}) that the hyperplanes of this $h$-dimensional vector space are in one-to-one correspondence with the sets $\{x\in \F_{q}\mid\Tr(\alpha x)=0\}$ where $\alpha\in \F_{q}^*$. Vector subspaces of codimension two can be written as the intersection of two different hyperplanes, which gives us the following lemma.

\begin{lemma}\label{viatrace} 
If $S$ is an additive subgroup of order $q/4$ of $\F_{q}$, $q=2^h$, then $S=\{ x\in \F_{q}\mid \Tr(\mu_1 x)=\Tr(\mu_2 x)=0\}$ for some $\mu_1\neq \mu_2\in \F_{q}^*$.
\end{lemma}

\begin{lemma}\label{S=alfaS} Let $S$ be an additive subgroup of order $q/4$ of $\F_{q}$, $q=2^h$. If $S=\alpha S$, for some $\alpha \in \F^{*}_{q}\setminus\{1\}$, then $\alpha \in \F_{4}$ and hence $h$ is even and $S=\{x\in\F_q\mid \Tr(\mu x)=\Tr(\alpha\mu x)=0\}$ for some $\mu\in\F_{q}$. Moreover, in this case for every $\beta \in \F_{q}\setminus\F_{4}$, we have $\langle S,\beta S\rangle=\F_{q}$ and for every $\beta \in\F^{*}_{4}$, we have $S=\beta S$.
\end{lemma}
\begin{proof} By Lemma \ref{viatrace}, we have that $S=\{ x\in \F_q\mid \Tr(\mu_1 x)=\Tr(\mu_2 x)=0\}$ for some $\mu_{1},\mu_{2}\in\F^{*}_{q}$, $\mu_{1}\neq\mu_{2}$. Clearly $\alpha S=\{x\in \F_q\mid\Tr((\mu_1/\alpha) x)=\Tr((\mu_2/\alpha) x)=0\}$. Suppose that $S=\alpha S$, for some $\alpha\neq 1$, then both $\mu_1$ and $\mu_2$ have to be contained in the set $\{\mu_1/\alpha,\mu_2/\alpha,(\mu_1+\mu_2)/\alpha\}$. If $\mu_1=\mu_1/\alpha$, then $\alpha=1$, a contradiction. 
\par Note that $\mu_{2}$ and $\mu_1+\mu_2$ can be interchanged, hence without loss of generality we may assume $\mu_1=\mu_2/\alpha$. Then, either $\mu_2=\mu_1/\alpha$ or $\mu_2=(\mu_1+\mu_2)/\alpha$ since $\mu_2=\mu_2/\alpha$ implies that $\alpha=1$. In the former case, we have that $\alpha^2=1$ and hence, $\alpha=1$, a contradiction. In the latter case, we have that $\alpha^2=\alpha+1$ and hence, $\alpha\in \F_{4}$ and $h$ is even. Also $S$ is given by $\{ x\in \F_q\mid \Tr(\mu_2 x)=\Tr(\alpha\mu_2 x)=0\}$.
\par Consider $\beta\in\F^{*}_{q}$. The subgroup $\beta S$ equals $\{x\in \F_q \mid \Tr((\mu_1/\beta) x)=\Tr((\mu_2/\beta) x)=0\}$. Now suppose that $\langle S,\beta S\rangle$ is a subgroup of order $q/2$ (or equivalently, defines a hyperplane of $\F_q$), then we have that the elements in the set $V=\{\mu_1/\beta,\mu_2/\beta, \mu_1,\mu_2\}$ are linearly dependent over $\F_2$. Since $\mu_1=\mu_2/\alpha$, it follows that the elements in $V=\{\mu_1/\beta,\alpha \mu_1/\beta,\mu_1,\alpha \mu_1\}$ are linearly dependent, and hence, since $\mu_1\neq 0$ and $\alpha^{-1}=\alpha^{2}=\alpha+1$, we know that $1/\beta$ is an $\F_2$-linear combination of $\alpha$ and $1$. It follows that $\beta \in \F_4^*$. If $\beta \in \F_{q}\setminus\F_{4}$, we conclude that $\langle S,\beta S\rangle$ cannot be a subgroup of order $q/2$, but by the first part of the proof it cannot have order $q/4$ either. Hence, $\langle S,\beta S\rangle$ equals $\F_{q}$. If $\beta=\alpha+1$, then $S=\alpha S$ implies that also $S=(\alpha+1) S=\beta S$. Hence, for every $\beta \in\F^{*}_{4}=\{1,\alpha,\alpha+1\}$, we have $S=\beta S$.
\end{proof}

\begin{lemma}\label{OKofF8}
	Let $S$ be an additive subgroup of order $q/4$  in $\F_{q}$, $q=2^h$, and let $\alpha,\beta\in\F_{q}$. If $\left\langle S,\alpha S\right\rangle$, $\left\langle S,\beta S\right\rangle$, $\left\langle \alpha S,\beta S\right\rangle$ and $\left\langle (\alpha+1)S,(\beta+1) S\right\rangle$ are subgroups of order $q/2$, then either $\beta=\alpha+1$ and $S=\{ x\in \F_{q} \mid \Tr(\gamma x)=\Tr(\alpha \gamma x)=0\}$ for some $\gamma\in\F_{q}$, or $3\mid h$, $\alpha,\beta \in \F_8\subseteq\F_{q}$ and there is an additive subgroup $S'$ of order $q/8$ in $\F_{q}$ such that $S'=S\cap\alpha S\cap\beta S\cap (\alpha +1)S\cap (\beta +1)S$. Moreover, if in this case $\left\langle S,\alpha S\right\rangle=\{ x\in \F_{q} \mid \Tr(\gamma x)=0\}$ for some $\gamma\in\F_{q}$, then $S'=\{x\in\F_{q}\mid\forall y\in\F_{8}:\Tr(xy\gamma)=0\}$.
\end{lemma}
\begin{proof} First note that the conditions of the lemma imply that $\alpha\neq\beta$ and that $\alpha,\beta\notin\{0,1\}$. For convenience, we consider the subgroups as subspaces of the projective space $\mathrm{PG}(\F_2^h)=\PG(h-1,2)$, and dualise. The dualisation is induced by the bijection that maps the hyperplane $\{x\mid \Tr(kx)=0\}$ onto the vector line $k$ and vice versa. By $T^D$ we denote the dual of a subspace $T$. The subspaces $S^D$, $(\alpha S)^D$ and $(\beta S)^D$ are vector planes of $\F_2^h$ hence, we consider them as lines of $\PG(h-1,2)$. %

The condition that the subgroup generated by any two elements of $\{S,\alpha S, \beta S\}$ has order $q/2$ translates into the condition that the lines $S^D$, $(\alpha S)^D$ and $(\beta S)^D$ mutually intersect in a point. Note that $\langle S,\alpha S\rangle$ always contains $(\alpha +1)S$, which implies that the line $((\alpha+1)S)^D$ goes through $S^D\cap (\alpha S)^D$, and similarly, the line $((\beta+1)S)^D$ goes through $S^D\cap (\beta S)^D$. We denote the hyperplane $\langle S,\alpha S\rangle$ by $\{ x\in \F_{q} \mid \Tr(\gamma x)=0\}$. Then, $S=\{ x\in \F_{q} \mid \Tr(\gamma x)=\Tr(\alpha \gamma x)=0\}$ and $\alpha S=\{ x\in \F_{q} \mid \Tr(\gamma x)=\Tr((\gamma/\alpha) x)=0\}$. Note that $S\neq\alpha S$ also implies that $\alpha^{2}\neq\alpha+1$ (see Lemma \ref{S=alfaS}).

Recall that the lines $S^D$, $(\alpha S)^D$ and $(\beta S)^D$ mutually intersect in a point. The first possibility is that the lines $S^D=\left\langle\gamma,\alpha\gamma\right\rangle$, $(\alpha S)^D=\left\langle\gamma,\frac{\gamma}{\alpha}\right\rangle$ and $(\beta S)^D=\left\langle\frac{\gamma}{\beta},\frac{\alpha\gamma}{\beta}\right\rangle$ go through a common point. Then, we have that $1/\beta=1$, $\alpha/\beta=1$ or $(\alpha+1)/\beta=1$. The first and the second lead immediately to a contradiction, so $\beta=\alpha+1$.

The second possibility is that the lines $S^D$, $(\alpha S)^D$ and $(\beta S)^D$ are contained in a common plane but are not concurrent. The plane spanned by $S^{D}$ and $(\alpha S)^{D}$ is the plane $\left\langle\gamma,\alpha\gamma,\frac{\gamma}{\alpha}\right\rangle$. Since $(\beta S)^D$ lies in this plane, $\frac{\gamma}{\beta}$ and $\frac{\alpha\gamma}{\beta}$ are both a linear combination of $\gamma$, $\alpha\gamma$ and $\frac{\gamma}{\alpha}$. Hence, $\frac{1}{\beta}$ and $\frac{\alpha}{\beta}$ are both a linear combination of $\alpha$, $1$ and $\frac{1}{\alpha}$. So, $\frac{\alpha}{\beta}$ is both a linear combination of $\alpha$, $1$ and $\frac{1}{\alpha}$ and of $\alpha^{2}$, $\alpha$ and $1$. Then, either $\frac{\alpha}{\beta}=\alpha+1$ or else there exist $\lambda_{1},\lambda_{2}\in\F_{2}$ such that $\alpha^{2}+\lambda_{2}\alpha+\lambda_{1}+\frac{1}{\alpha}=0$, equivalently such that $\alpha^{3}+\lambda_{2}\alpha^{2}+\lambda_{1}\alpha+1=0$. In the latter case we must have that $\alpha^3+\alpha+1=0$ or $\alpha^3+\alpha^2+1=0$, since $(\alpha+1)^3\neq0$. Hence $\alpha \in \F_8$ and consequently also $\beta\in\F_{8}$ and $3\mid h$.
\par In the former case we have that $\beta=\frac{\alpha}{\alpha+1}$. Then, $((\beta+1)S)^{D}=\left\langle(\alpha+1)\gamma,\alpha(\alpha+1)\gamma\right\rangle$. We also have that $(\alpha+1)S=\left\langle\gamma, \frac{\gamma}{\alpha+1}\right\rangle$.  Since $((\alpha+1)S)^{D}$ and $((\beta+1)S)^{D}$ have a point in common the sets $\{\alpha+1,\alpha^{2}+\alpha,\alpha^{2}+1\}$ and $\left\{1,\frac{1}{\alpha+1},\frac{\alpha}{\alpha+1}\right\}$ have an element in common. If follows that either $\alpha^{2}=\alpha+1$, in which case $\beta=\alpha+1$, or else $\alpha^{3}+\alpha^{j}+1=0$ with $j\in\{1,2\}$, in which case $\alpha \in \F_8$ and hence also $\beta\in\F_{8}$ and $3\mid h$.
\par Note that in the cases with $\alpha,\beta\in\F_{8}$ we have that the lines $S^{D}$, $(\alpha S)^{D}$, $(\beta S)^{D}$, $((\alpha+1) S)^{D}$, $((\beta+1) S)^{D}$ are contained in the plane $\left\langle\gamma,\alpha\gamma,\frac{\gamma}{\alpha}\right\rangle=\left\langle\gamma,\alpha\gamma,\alpha^{2}\gamma\right\rangle$. Hence, the subgroup $\left\langle\gamma,\alpha\gamma,\alpha^{2}\gamma\right\rangle^{D}=S'=\{x\in\F_{q}\mid\forall y\in\F_{8}:\Tr(xy\gamma)=0\}$ equals the intersection $S\cap\alpha S\cap\beta S\cap (\alpha +1)S\cap (\beta +1)S$.
\end{proof}

\begin{lemma}\label{nosolutions} Let $\F_{q}$ be a field that admits $\F_{8}$ as a subfield. Let $\alpha,\beta\in \F_8\setminus\{0,1\}$ with $\beta \notin\{\alpha,\alpha+1\}$ and let $S$ be an additive subgroup of order $q/4$ of $\F_{q}$. Assume that $\langle S,\alpha S\rangle=\{x\in \F_q\mid \Tr(k_1 x)=0\}$, $\langle S,\beta S\rangle=\{x\in \F_q\mid \Tr(k_2 x)=0\}$, $\langle \alpha S,\beta S\rangle=\{x\in \F_q\mid \Tr(k_3 x)=0\}$ and $\langle (\alpha+1) S,(\beta+1) S\rangle=\{x\in \F_q\mid \Tr(k_4 x)=0\}$ for certain pairwise different $k_{1},k_{2},k_{3},k_{4}\in\F^{*}_{q}$. Then the system of equations
\begin{numcases}{}
\Tr(k_1(X+\alpha))=1\label{eq1}\\
\Tr(k_2(Y+\beta))=1\label{eq2}\\
\Tr(k_3(\alpha Y+\beta X))=1\label{eq3}\\
\Tr(k_4((\alpha+\beta)+(\beta+1)X+(\alpha+1)Y))=1\label{eq4}
\end{numcases}
has no solutions $(X,Y)\in \F_q^2$.
\end{lemma}
\begin{proof}
	Using the same dualisation and notation as in the proof of Lemma \ref{OKofF8}, we consider the lines $S^D,(\alpha S)^D$, $(\beta S)^D$, $((\alpha+1)S)^D$ and $((\beta+1)S)^D$. It follows from the conditions of the lemma that these lines are all different. From Lemma \ref{OKofF8} we know that these lines are contained in a plane $\pi$, the dual of the subgroup $\{x\mid\forall y\in \F_{8}:\Tr(k_{1}xy)=0\}$. As the line $((\alpha+1)S)^D$ goes through $S^D\cap (\alpha S)^D$, it is the unique third line in $\pi$ through $S^D\cap (\alpha S)^D$, different from $S^D$ and $(\alpha S)^D$. Since $\langle S,\alpha S\rangle=\{x\in \F_q\mid \Tr(k_1 x)=0\}$, the point $S^D\cap (\alpha S)^D$ is the point $k_1$. Similarly, the line $((\beta+1)S)^D$ is the unique line which goes through $S^D\cap (\beta S)^D=k_2$, and is different from $S^D$ and $(\beta S)^D$. The lines $(\alpha S)^D$ and $(\beta S)^D$ meet in the point $k_3$. The plane $\pi$ is generated by $k_{1}$, $k_{1}$ and $k_{3}$. It follows that the point $k_4$, which is the intersection of $((\alpha+1)S)^D$ and $((\beta+1)S)^D$ (and also is contained in $((\alpha+\beta)S)^D$) is equal to $k_1+k_2+k_3$, as it is the unique point of the plane, not on the lines $S^{D}$, $(\alpha S)^{D}$ and $(\beta S)^{D}$.
	\par Since $\langle S,\alpha S\rangle=\{x\in \F_q\mid \Tr(k_1 x)=0\}$ we know that $S=\{ x\in \F_{q} \mid \Tr(k_{1} x)=\Tr(\alpha k_{1} x)=0\}$ and $\alpha S=\{ x\in \F_{q} \mid \Tr(k_{1}x)=\Tr((k_{1}/\alpha) x)=0\}$. Analogously it follows from $\langle S,\beta S\rangle=\{x\in \F_q\mid \Tr(k_2 x)=0\}$ that $S=\{ x\in \F_{q} \mid \Tr(k_{2} x)=\Tr(\beta k_{2} x)=0\}$ and $\beta S=\{ x\in \F_{q} \mid \Tr(k_{2}x)=\Tr((k_{2}/\beta) x)=0\}$. Comparing the expressions for $S$, This implies that
	\begin{align}\label{eq5}
	k_2/k_1 \in \{\alpha,\alpha+1\}
	\end{align}
and 
\begin{align}\label{eq6}
k_2/k_1\in \{\alpha/\beta,1/\beta,(\alpha+1)/\beta\}\;.
\end{align}
Similarly, from $\langle\alpha S,\beta S\rangle=\{x\in \F_q\mid \Tr(k_3 x)=0\}$ it follows that $\alpha S=\{ x\in \F_{q} \mid \Tr(k_{3} x)=\Tr((\beta/\alpha) k_{3} x)=0\}$ and $\beta S=\{ x\in \F_{q} \mid \Tr(k_{3}x)=\Tr((\alpha/\beta)k_{3} x)=0\}$, and we find that
\begin{align}\label{eq7}
k_3/k_1 \in\{1/\alpha,(\alpha+1)/\alpha\}
\end{align}
and
\begin{align}\label{eq8}
k_3/k_1\in \{\alpha/\beta,1/\beta,(\alpha+1)/\beta\}.
\end{align}
Looking at the system of equations, we see that a solution $(X,Y)$ to \eqref{eq1} and \eqref{eq2} is of the form $\left(\frac{t}{k_1}+\alpha,\frac{t'}{k_2}+\beta\right)$ for some $t,t'\in \F_q$ with $\Tr(t)=\Tr(t')=1$. Equation \eqref{eq3} gives us that 
\begin{align}\label{eq9}
\frac{k_3}{k_2}\alpha t'+\frac{k_3}{k_1}\beta t=t''
\end{align}
for some $t''$ with $\Tr(t'')=1$.
Finally, equation \eqref{eq4} with $k_4=k_1+k_2+k_3$ yields
\begin{align}\label{eq10}
(k_1+k_2+k_3)\left(\frac{\beta+1}{k_1}t+\frac{\alpha+1}{k_2}t'\right)=t'''
\end{align}
for some $t'''\in \F_q$ with $\Tr(t''')=1$.

We know that $\alpha^3=\alpha+1$ or $\alpha^{3}=\alpha^2+1$ and $\beta\in \{\alpha^2,\alpha^2+1,\alpha^2+\alpha,\alpha^2+\alpha+1\}$. Suppose first that $\alpha^3=\alpha+1$ and $\beta=\alpha^2$. It follows from equations \eqref{eq5} and \eqref{eq6} that $k_2/k_1=\alpha=(\alpha+1)/\beta$. From equations \eqref{eq7} and \eqref{eq8}, we get that $k_3/k_1=1/\alpha=\alpha/\beta$. Hence, $k_4=k_1(1+\alpha+1/\alpha)$.
Equation \eqref{eq9} becomes $\alpha t+(\alpha^2+1)t'=t''$ while equation \eqref{eq10} becomes $(\alpha+1)t+(\alpha^2+1)t'=t'''$. This implies that $t+t''=t'''$, a contradiction since $\Tr(t+t'')=0$ and $\Tr(t''')=1$.

A tedious calculation shows the following results.
\begin{itemize}
\item If $\alpha^3=\alpha+1$ and $\beta=\alpha^2+1$, then $k_2/k_1=\alpha$, $k_3/k_1=(\alpha+1)/\alpha$.
\item If $\alpha^3=\alpha+1$ and $\beta=\alpha^2+\alpha$, then $k_2/k_1=\alpha+1$, $k_3/k_1=1/\alpha$.
\item If $\alpha^3=\alpha+1$ and $\beta=\alpha^2+\alpha+1$, then $k_2/k_1=\alpha+1$, $k_3/k_1=(\alpha+1)/\alpha$.
\item If $\alpha^3=\alpha^2+1$ and $\beta=\alpha^2$, then $k_2/k_1=\alpha+1$, $k_3/k_1=1/\alpha$.
\item If $\alpha^3=\alpha^2+1$ and $\beta=\alpha^2+1$, then $k_2/k_1=\alpha+1$, $k_3/k_1=(\alpha+1)/\alpha$.
\item If $\alpha^3=\alpha^2+1$ and $\beta=\alpha^2+\alpha$, then $k_2/k_1=\alpha$, $k_3/k_1=1/\alpha$.
\item If $\alpha^3=\alpha^2+1$ and $\beta=\alpha^2+\alpha+1$, then $k_2/k_1=\alpha$, $k_3/k_1=(\alpha+1)/\alpha$.
\end{itemize}
In all of the above cases, the reader can check that plugging these values in equations \eqref{eq9} and \eqref{eq10} gives a contradiction in the same way as deduced above.
\end{proof}

\begin{theorem}\label{elationLinPencil} If a KM-arc $\mathcal{A}$ of type $q/4$ in $\PG(2,q)$ is an elation KM-arc, then its $q/4$-secants define an $\F_2$-linear pencil with head corresponding to the elation line. Moreover, if the elation line is given by $X=0$ and the linear pencil of $q/4$-secants is given by the set of lines with equation $Y=kX$ with $k\in\left\langle\alpha,1\right\rangle$ up to $\PGL$-equivalence, then the subgroup determined by the points of $\mathcal{A}$ on its $q/4$-secants is given by $S=\{x\mid\Tr(\mu x)=\Tr(\alpha\mu x)=0\}$ for some $\mu\in\F_{q}$.
\end{theorem}
\begin{proof} By Lemma \ref{vormq/4}, we know that up to $\PGL$-equivalence we can take the elation line to be $X=0$ and the affine points of $\mathcal{A}$ to be $\{(1,0,s)\mid s\in S\}\cup\{(1,1,1+s)\mid s\in S\}\cup\{(1,\alpha,\alpha'+s)\mid s\in S\}\cup\{(1,\beta,\beta'+s)\mid s\in S\}$ for some $\alpha,\alpha',\beta,\beta'\in\F_{q}$, where $S$ has order $q/4$. For any $a\in\F_{q}$ we denote the line $Y=aX$ by $\ell_{a}$. Then we see that the affine points are contained in the lines $\ell_0,\ell_1,\ell_\alpha$ and $\ell_\beta$.

Three points of $\mathcal{A}$ on the lines $\ell_0,\ell_1,\ell_\alpha$ respectively, cannot be collinear. Hence, we find that $\left|\begin{smallmatrix} 1&0&s_1\\1&1&1+s_2\\1&\alpha&\alpha'+s_3\end{smallmatrix}\right|\neq 0$ for all $s_1,s_2,s_3\in S$. This implies that for all $s_1,s_2,s_3\in S$, $\alpha'+s_3+\alpha+\alpha s_2+\alpha s_1+s_1\neq 0$, and hence, that $\alpha+\alpha'\notin \langle S,\alpha S\rangle$. In particular, $\langle S,\alpha S\rangle$ cannot be the entire field $\F_q$. In a similar way, by looking at three points on $\ell_0,\ell_1,\ell_\beta$, we find that $\beta+\beta'\notin\langle S,\beta S\rangle$, hence, $\langle S,\beta S\rangle\neq \F_q$, and by looking at points on $\ell_0,\ell_\alpha,\ell_\beta$, that $\alpha \beta'+\beta\alpha'\notin\langle \alpha S,\beta S\rangle$, and hence $\langle \alpha S,\beta S\rangle$ is not $\F_q$.
For three points on $\ell_1,\ell_\alpha,\ell_\beta$, we find that for all $s_1,s_2,s_3$ in $\F_q$,
\begin{align*}
&\alpha \beta'+\beta \alpha'+\alpha'+\beta'+\alpha+\beta+\alpha(s_1+s_3)+\beta(s_1+s_2)+s_2+s_3\neq0\\
\Leftrightarrow\quad&\alpha \beta'+\beta \alpha'+\alpha'+\beta'+\alpha+\beta+(\alpha+1)(s_1+s_3)+(\beta+1)(s_1+s_2)\neq 0\;.
\end{align*}
%The expression $\alpha(s_1+s_3)+\beta(s_1+s_2)+s_2+s_3$ is equal to $\alpha(s_1+s_3)+\beta(s_1+s_2)+(s_1+s_2)+(s_1+s_3)$, and
Hence, we have that $\alpha \beta'+\beta \alpha'+\alpha'+\beta'+\alpha+\beta\notin\langle (\alpha+1)S,(\beta+1)S\rangle$ and so $\langle (\alpha+1)S,(\beta+1)S\rangle$ cannot be $\F_q$.

Since $S,\alpha S, \beta S, (\alpha+1)S$ and $(\beta+1)S$ are additive subgroups of $\F_{q}$ of size $q/4$, by Lemmas \ref{S=alfaS} and \ref{OKofF8}, we find that either $\beta=\alpha+1$, and then $\ell_0,\ell_1,\ell_\alpha,\ell_\beta,\ell_\infty$ define an $\F_2$-linear pencil with $\ell_\infty$ as head, or $3| h$ and $\alpha,\beta \in \F_8$. 
Suppose we are in the latter case. Since $\alpha+\alpha' \notin \langle S,\alpha S\rangle$, we have that $\Tr(k_1(\alpha+\alpha'))=1$ with $k_{1}$ such that  $\langle S,\alpha S\rangle$ is the subgroup of all elements $\{x\mid\Tr(k_1x)=0\}$. Similarly, $\Tr(k_2(\beta+\beta'))=1$, with $k_{2}$ such that $\langle S,\beta S\rangle=\{x\mid\Tr(k_2x)=0\}$, and $\Tr(k_3(\alpha\beta'+\beta \alpha'))=1$, with $k_{3}$ such that $\langle \alpha S,\beta S\rangle=\{x\mid\Tr(k_3x)=0\}$, and $\Tr(k_4((\alpha+1)\beta'+(\beta+1)\alpha'+\alpha+\beta))=1$, with $k_{4}$ such that $\langle (\alpha+1)S,(\beta+1)S\rangle=\{x\mid\Tr(k_4x)=0\}$. But by Lemma \ref{nosolutions}, there is no solution $(\alpha',\beta')$ for this system of equations.
\par So, $\beta=\alpha+1$ and the $q/4$-secants determine an $\F_{2}$-linear pencil with $\ell_{\infty}$ as head. Either $S=\alpha S$ or $\left\langle S,\alpha S\right\rangle$ has order $q/2$. In the former case, the second part of the statement follows form Lemma \ref{S=alfaS} and in the latter case it follows from Lemma \ref{OKofF8}.
\end{proof}

\begin{remark}\label{remlinearpencil}
	We believe that the statement of Theorem \ref{elationLinPencil} holds for general elation KM-arcs of type $t$, i.e. that the $t$-secants to an elation KM-arc of type $t$ define an $\F_2$-linear pencil (with the elation line as head). It is worth mentioning that this property also seems to hold for elation hyperovals, where the pencil that should be $\F_2$-linear is the set of $2$-secants through the centre of the non-trivial elation.
\end{remark}

\begin{theorem}\label{mainq/4}
	Let $\A$ be an elation KM-arc of type $q/4$, then $\A$ is $\PGL$-equivalent to the KM-arc $\mathcal{A}_{1/\beta^2,\beta,0,0,0}$. Hence, $\A$ is a translation KM-arc.
\end{theorem}
\begin{proof}
	By Theorem \ref{elationLinPencil}, we know that the $q/4$-secants to an elation KM-arc of type $q/4$ define an $\F_2$-linear pencil with head the elation line. Hence, by Lemmas \ref{vormq/4} and \ref{elationLinPencil} $\A$ is equivalent to a KM-arc with elation line $X=0$ and affine point set $\{(1,0,s)\mid s\in S\}\cup\{(1,1,1+s)\mid s\in S\}\cup\{ (1,\beta,\alpha_{1}+s)\mid s\in S\}\cup \{(1,\beta+1,\alpha_{2}+s)\mid s\in S\}$, where $\alpha_{1},\alpha_{2},\beta$ are elements of $\F_{q}$ and $S$ is an additive subgroup of order $q/4$ of $\F_{q}$ given by $S=\{x\mid\Tr(\mu x)=\Tr(\beta\mu x)=0\}$ for some $\mu\in\F_{q}$. We see that the KM-arc $\mathcal{A}$ is $\PGL$-equivalent with the KM-arc $\mathcal{A}_{1/\beta^2,\beta,0,0}$.
	\par The statement now follows from Theorem \ref{transliff} or Lemma \ref{iffTranslation}.
\end{proof}

\section{A new family of elation KM-arcs of type \texorpdfstring{$q/8$}{q/8}}\label{q/8}

We start by recalling the definition of the Kronecker delta.

\begin{definition}
	For two integers $i$ and $j$, the Kronecker delta $\delta_{i,j}$ equals $1$ if $i=j$ and $0$ else.
\end{definition}

We can consider the Kronecker delta as a $\Z\times\Z\to\Z$ function that maps $(i,j)$ onto $\delta_{i,j}$. We now define a similar function for vectors over $\F_{2}$.

\begin{definition}
	The function $M^{k}_{n}:(\F^{k}_{2})^{n}\to\F_{2}$ is the function taking $n$ vectors of length $k$ as argument and mapping them to 0 if two of these vectors are equal and to 1 otherwise.
\end{definition}
The proof of the following lemma is left to the reader.
\begin{lemma}\label{Ms}
	Let $x,y,z$ be vectors in $\F^{k}_{2}$.
	\begin{enumerate}[label=(\roman*)]
		\item $M^{k}_{2}(x,y)=1+\prod^{k}_{i=1}(x_{i}+y_{i}+1)$
		\item $M^{k}_{3}(x,y,z)=M^{k}_{2}(x,y)+M^{k}_{2}(y,z)+M^{k}_{2}(z,x)$
	\end{enumerate}
\end{lemma}

We now construct a new family of KM-arcs.

\begin{theorem}\label{consq8}
	Let $q=2^h$, $h\geq4$. Let $\alpha_{1},\alpha_{2},\alpha_{3}\in\F^{*}_{q}$ be $\F_{2}$-independent and define $S=\{x\in\F_{q}\mid \forall i:\Tr(\alpha_{i}x)=0\}$. Let  $\beta_{1},\beta_{2},\beta_{3}\in\F^{*}_{q}$ be such that $\Tr(\alpha_{i}\beta_{j})=\delta_{i,j}$. Let $f_{1},f_{2},f_{3}$ be the three functions $\F^{3}_{2}\to\F_{2}$%, where we consider $\F_{2}$ as a subset of $\F_{q}$
	, given by $f_{1}:(x,y,z)\mapsto x+y+z+yz$, $f_{2}:(x,y,z)\mapsto y+z+xz$ and $f_{3}:(x,y,z)\mapsto z+xy$.
	\par For any $(\lambda_{1},\lambda_{2},\lambda_{3})\in\F^{3}_{2}$ we define
	\[
	\mathcal{S}_{(\lambda_{1},\lambda_{2},\lambda_{3})}=\left\{\left(1,\sum^{3}_{i=1}\lambda_{i}\alpha_{i},\sum^{3}_{i=1}f_{i}(\lambda_{1},\lambda_{2},\lambda_{3})\beta_{i}+s\right)\Bigg|\ s\in S\right\}\;.
	\]
	We also define $\mathcal{S}_{0}=\{(0,1,x)\mid\forall i: \Tr(\alpha^{2}_{i}x)=0\}$ and
	$\mathcal{A}=\mathcal{S}_{0}\cup\bigcup_{v\in\F^{3}_{2}}\mathcal{S}_{v}$. If $q>16$, then $\mathcal{A}$ is an elation KM-arc of type $q/8$ in $\PG(2,q)$ with elation line $X=0$ and $q/8$-nucleus $(0,0,1)$. If $q=16$, then $\mathcal{A}$ is an elation hyperoval in $\PG(2,q)$ with elation line $X=0$.
\end{theorem}
\begin{proof}
	We know that $S$ is a subgroup of $\F_{q},+$ containing $q/8$ elements. Now note that the element $\beta_{j}$, $j=1,2,3$ is a coset leaders of the coset $\{x\in\F_{q}\mid \forall i:\Tr(\alpha_{i}x)=\delta_{i,j}\}$ which also implies that the existence of elements $\beta_{1},\beta_{2},\beta_{3}$ is guaranteed.
	\par It is immediate that the points of $\mathcal{S}_{(\lambda_{1},\lambda_{2},\lambda_{3})}$, with $(\lambda_{1},\lambda_{2},\lambda_{3})\in\F^{3}_{2}$, are on the line $\ell_{(\lambda_{1},\lambda_{2},\lambda_{3})}$ with equation $(\sum^{3}_{i=1}\lambda_{i}\alpha_{i})X+Y=0$ and that the points of $\mathcal{S}_{0}$ are on the line $\ell_{\infty}$ with equation $X=0$. Hence, all lines through $N(0,0,1)$ either contain $q/8$ or $0$ points of $\mathcal{A}$.
	\par Now we check that three points on different $q/8$-secants are not collinear. First we assume that $\ell_{\infty}$ is not among these three $q/8$-secants. Then the three points can be described as
	\[
	\left(1,\sum^{3}_{i=1}\lambda_{i}\alpha_{i},\sum^{3}_{i=1}f_{i}(\overline{\lambda})\beta_{i}+s\right), \left(1,\sum^{3}_{i=1}\lambda'_{i}\alpha_{i},\sum^{3}_{i=1}f_{i}(\overline{\lambda}')\beta_{i}+s'\right)\text{ and } \left(1,\sum^{3}_{i=1}\lambda''_{i}\alpha_{i},\sum^{3}_{i=1}f_{i}(\overline{\lambda}'')\beta_{i}+s''\right)
	\] 
	with $\overline{\lambda}=(\lambda_{1},\lambda_{2},\lambda_{3})$, $\overline{\lambda}'=(\lambda'_{1},\lambda'_{2},\lambda'_{3})$ and $\overline{\lambda}''=(\lambda''_{1},\lambda''_{2},\lambda''_{3})$ three pairwise different vectors in $\F^{3}_{2}$. We find that
	\begin{align*}
	\Delta&=\begin{vmatrix}
	1 & \sum^{3}_{i=1}\lambda_{i}\alpha_{i} & \sum^{3}_{i=1}f_{i}(\overline{\lambda})\beta_{i}+s\\
	1 & \sum^{3}_{i=1}\lambda'_{i}\alpha_{i} & \sum^{3}_{i=1}f_{i}(\overline{\lambda}')\beta_{i}+s'\\
	1 & \sum^{3}_{i=1}\lambda''_{i}\alpha_{i} & \sum^{3}_{i=1}f_{i}(\overline{\lambda}'')\beta_{i}+s''
	\end{vmatrix}\\
	&=\sum_{cyc}\left(\sum_{i,j=1}^{3}(\lambda_{i}f_{j}(\overline{\lambda}')+\lambda'_{i}f_{j}(\overline{\lambda}))\alpha_{i}\beta_{j}+s\sum^{3}_{i=1}\lambda'_{i}\alpha_{i}+s'\sum^{3}_{i=1}\lambda_{i}\alpha_{i}\right)
	\end{align*}
	where the cyclic sum is taken over $(\overline{\lambda},\overline{\lambda}',\overline{\lambda}'')$ and the corresponding $(s,s',s'')$. We calculate the trace of both sides of this equation. Considering that $\Tr(\alpha_{i}t)=0$ for all $t\in S$, $i=1,2,3$, that $\Tr(\alpha_{i}\beta_{j})=\delta_{i,j}$ and that the trace function is $\F_{2}$-linear, we find that
	\begin{align}\label{m3}
	\Tr(\Delta)&=\sum_{cyc}\left(\sum_{i,j=1}^{3}(\lambda_{i}f_{i}(\overline{\lambda}')+\lambda'_{i}f_{i}(\overline{\lambda}))\right)\nonumber\\
	&=\sum_{cyc}\left(\lambda_{1}(\lambda'_{1}+\lambda'_{2}+\lambda'_{3}+\lambda'_{2}\lambda'_{3})+\lambda'_{1}(\lambda_{1}+\lambda_{2}+\lambda_{3}+\lambda_{2}\lambda_{3})+\lambda_{2}(\lambda'_{2}+\lambda'_{3}+\lambda'_{1}\lambda'_{3})\right.\nonumber\\
	&\qquad\qquad\left.+\lambda'_{2}(\lambda_{2}+\lambda_{3}+\lambda_{1}\lambda_{3})+\lambda_{3}(\lambda'_{3}+\lambda'_{1}\lambda'_{2})+\lambda'_{3}(\lambda_{3}+\lambda_{1}\lambda_{2})\right)\nonumber\\
	&=\sum_{cyc}\left(M^{3}_{2}(\overline{\lambda},\overline{\lambda}')+\sum_{i=1}^{3}(\lambda_{i}+1)+\sum_{i=1}^{3}(\lambda'_{i}+1)+1\right)\nonumber\\
	&=M^{3}_{2}(\overline{\lambda},\overline{\lambda}')+M^{3}_{2}(\overline{\lambda}',\overline{\lambda}'')+M^{3}_{2}(\overline{\lambda}'',\overline{\lambda})\nonumber\\
	&=M^{3}_{3}(\overline{\lambda},\overline{\lambda}',\overline{\lambda}'')\;.
	\end{align}
	In the last step, we used Lemma \ref{Ms}(ii).
	It follows that $\Delta\neq0$ because the vectors $\overline{\lambda}$, $\overline{\lambda}'$ and $\overline{\lambda}''$ are pairwise different, hence the three points that we considered are not collinear.
	\par Now we assume that $\ell_{\infty}$ is among the three $q/8$-secants. Then, the three points can be described as
	\[
	\left(1,\sum^{3}_{i=1}\lambda_{i}\alpha_{i},\sum^{3}_{i=1}f_{i}(\overline{\lambda})\beta_{i}+s\right), \left(1,\sum^{3}_{i=1}\lambda'_{i}\alpha_{i},\sum^{3}_{i=1}f_{i}(\overline{\lambda}')\beta_{i}+s'\right)\text{ and } \left(0,1,t\right)
	\] 
	with $\overline{\lambda}=(\lambda_{1},\lambda_{2},\lambda_{3})$, $\overline{\lambda}'=(\lambda'_{1},\lambda'_{2},\lambda'_{3})$ two different vectors in $\F^{3}_{2}$ and $\Tr(\alpha^{2}_{i}t)=0$ for $i=1,2,3$. We find that
	\begin{align*}
	\Delta'&=\begin{vmatrix}
	0 & 1 & t\\
	1 & \sum^{3}_{i=1}\lambda_{i}\alpha_{i} & \sum^{3}_{i=1}f_{i}(\overline{\lambda})\beta_{i}+s\\
	1 & \sum^{3}_{i=1}\lambda'_{i}\alpha_{i} & \sum^{3}_{i=1}f_{i}(\overline{\lambda}')\beta_{i}+s'
	\end{vmatrix}=t\sum^{3}_{i=1}(\lambda_{i}+\lambda'_{i})\alpha_{i}+\sum_{i=1}^{3}(f_{i}(\overline{\lambda})+f_{i}(\overline{\lambda}'))\beta_{i}+s+s'\;.
	\end{align*}
	It follows that
	\begin{align}
	\Tr\left(\left(\sum_{i=1}^{3}(\lambda_{i}+\lambda'_{i})\alpha_{i}\right)\Delta'\right)&=\Tr\left(t\sum^{3}_{i=1}(\lambda_{i}+\lambda'_{i})\alpha^{2}_{i}+\sum_{i,j=1}^{3}(f_{j}(\overline{\lambda})+f_{j}(\overline{\lambda}'))(\lambda_{i}+\lambda'_{i})\alpha_{i}\beta_{j}\right.\nonumber\\
	&\qquad\qquad\left.+(s+s')\left(\sum_{i=1}^{3}(\lambda_{i}+\lambda'_{i})\alpha_{i}\right)\right)\nonumber\\
	&=\sum_{i=1}^{3}(f_{i}(\overline{\lambda})+f_{i}(\overline{\lambda}'))(\lambda_{i}+\lambda'_{i})\nonumber\\
	&=(\lambda_{1}+\lambda'_{1})(\lambda_{1}+\lambda_{2}+\lambda_{3}+\lambda_{2}\lambda_{3}+\lambda'_{1}+\lambda'_{2}+\lambda'_{3}+\lambda'_{2}\lambda'_{3})\nonumber\\
	&\quad+(\lambda_{2}+\lambda'_{2})(\lambda_{2}+\lambda_{3}+\lambda_{1}\lambda_{3}+\lambda'_{2}+\lambda'_{3}+\lambda'_{1}\lambda'_{3})\nonumber\\
	&\quad+(\lambda_{3}+\lambda'_{3})(\lambda_{3}+\lambda_{1}\lambda_{2}+\lambda'_{3}+\lambda'_{1}\lambda'_{2})\nonumber\\
	&=M^{3}_{2}(\overline{\lambda},\overline{\lambda}')\;.\label{veellambdas}
	\end{align}
	In the final step we used that all elements of $\F_{2}$ equal their square. Again we find that $\Delta'\neq0$, hence the three points are not collinear.
	\par We conclude that all lines not through $N$ contain at most two points of $\mathcal{A}$. For any point $P\in\mathcal{A}$ there are $q$ points of $\mathcal{A}$ not on the $q/8$-secant $\ell_{P}=\left\langle P,N\right\rangle$ so all $q$ lines through $P$ different from $\ell_{P}$ contain precisely two points of $\mathcal{A}$. Consequently, all lines of $\PG(2,q)$ contain 0, 2 or $q/8$ points of $\mathcal{A}$. So, $\mathcal{A}$ is a KM-arc of type $q/8$. From its definition and Lemmas \ref{Sinvoeren} and \ref{Sinvoeren2} it follows immediately that $\mathcal{A}$ is an elation KM-arc with elation line $X=0$ if $q>16$ and that $\mathcal{A}$ is an elation hyperoval with elation line $X=0$ if $q=16$.
\end{proof}

\begin{corollary}\label{q8exist}
	A KM-arc of type $q/8$ in $\PG(2,q)$, $q$ even, exists for all $q\geq16$.
\end{corollary}

This result follows immediately from the preceding theorem. The existence of KM-arcs of type $q/8$ was previously not generally known. We will discuss this in detail in Remark \ref{q8detail} 

\begin{remark}
	Instead of the three $\F^{3}_{2}\to\F_{2}$ functions $f_{1}:(x,y,z)\mapsto x+y+z+yz$, $f_{2}:(x,y,z)\mapsto y+z+xz$ and $f_{3}:(x,y,z)\mapsto z+xy$ that we used in Theorem \ref{consq8} we could have used other $\F^{3}_{2}\to\F_{2}$ functions. E.g., $f_{1}:(x,y,z)\mapsto y+z+yz$, $f_{2}:(x,y,z)\mapsto z+xz$ and $f_{3}:(x,y,z)\mapsto xy$ work as well. We chose the current representation because it also has a neat description of the points on the elation line.
	%andere veeltermen, kortere beschrijving
\end{remark}

We mention an interesting property on this class of KM-arcs.

\begin{theorem}\label{linpenq/8}
	The $q/8$-secants of the elation KM-arc of type $q/8$ constructed in Theorem \ref{consq8} define an $\F_2$-linear pencil, $q>16$.
\end{theorem}

It would be interesting to know whether Theorem \ref{elationLinPencil} is valid for all KM-arcs of type $q/8$ (see also Remark \ref{remlinearpencil}).

In Theorem \ref{nottranslation} we will give a negative answer to the question whether there are translation KM-arcs contained in the family of KM-arcs constructed in Theorem \ref{consq8}, but in order to prove this, we need some lemmas.
\begin{lemma}\label{nottrans}
	The KM-arc constructed in Theorem \ref{consq8} is not a translation KM-arc with the elation line as translation line.
\end{lemma}
\begin{proof}
	We use the notation introduced in the statement of Theorem \ref{consq8}. The elation line $\ell$ is given by $X=0$. We can see that $P_{1}(1,0,0)$, $P_{2}(1,\alpha_{1},\beta_{1})$ and $P(1,\alpha_{2},\beta_{1}+\beta_{2})$ are points of $\mathcal{A}$. The unique translation $\tau$ with translation line $\ell$ mapping $P_{1}$ onto $P_{2}$ is given by $\left(\begin{smallmatrix}
	1&0&0\\ \alpha_{1}&1&0\\ \beta_{1}&0&1
	\end{smallmatrix}\right)$. Then $P^{\tau}$ is $(1,\alpha_{1}+\alpha_{2},\beta_{2})$. From the construction it follows that all points of $\mathcal{A}$ on the line $Y+\left(\alpha_{1}+\alpha_{2}\right)X=0$ can be written as $(1,\alpha_{1}+\alpha_{2},t)$ with $\Tr(\alpha_{1}t)=0$ and  $\Tr(\alpha_{2}t)=\Tr(\alpha_{3}t)=1$. Since $\Tr(\alpha_{3}\beta_{2})=0$, the point $P^{\tau}$ is not in $\mathcal{A}$. Hence, $\mathcal{A}$ is not a translation KM-arc.
\end{proof}

Instead of this direct proof we could have applied \cite[Theorem 2.2]{wij}, but that would not have made the calculations easier.
\comments{
\begin{corollary}\label{notalltranslation}
	For any $q$ the family of KM-arcs in $\PG(2,q)$ constructed in Theorem \ref{consq8} contains elation KM-arcs which are not translation KM-arcs.
\end{corollary}
\begin{proof}
	It follows from Lemma \ref{nottrans} that the KM-arcs of this family are not translation KM-arcs with the elation line as translation line. If a KM-arc of type $t>2$ of this family is a translation KM-arc with a $t$-secant different from the elation line as translation line, then the subgroups $\{x\in\F_{q}\mid \forall i:\Tr(\alpha_{i}x)=0\}$ and $\{x\in\F_{q}\mid \forall i:\Tr(\alpha^{2}_{i}x)=0\}$ should coincide, hence the subgroups $\langle\alpha_{1},\alpha_{2},\alpha_{3}\rangle$ and $\langle\alpha^{2}_{1},\alpha^{2}_{2},\alpha^{2}_{3}\rangle$ should coincide which is in general not the case.
\end{proof}
}

It is clear that for any set $\{\alpha_{1},\alpha_{2},\alpha_{3}\}\subset\F_{q}$ that is an $\F_{2}$-independent triple, we can construct a KM-arc in $\PG(2,q)$ through Theorem \ref{consq8}. However some of the obtained KM-arcs will be $\PGammaL$-equivalent. 

We first prove that the construction in Theorem \ref{consq8} only depends on the subgroup $\left\langle\alpha_{1},\alpha_{2},\alpha_{3}\right\rangle$ and not on the choice of $\alpha_{1},\alpha_{2},\alpha_{3}$.

\begin{lemma}\label{projequiv1}
	Let $\{\alpha_{1},\alpha_{2},\alpha_{3}\}\subset\F^{*}_{q}$ and $\{\alpha'_{1},\alpha'_{2},\alpha'_{3}\}\subset\F^{*}_{q}$ be both $\F_{2}$-independent sets such that $\left\langle\alpha_{1},\alpha_{2},\alpha_{3}\right\rangle=\left\langle\alpha'_{1},\alpha'_{2},\alpha'_{3}\right\rangle$. Let $\mathcal{A}$ be the KM-arc constructed through Theorem \ref{consq8} using the triple $(\alpha_{1},\alpha_{2},\alpha_{3})$ and let $\mathcal{A}'$ be the KM-arc constructed through Theorem \ref{consq8} using the triple $(\alpha'_{1},\alpha'_{2},\alpha'_{3})$. Then $\mathcal{A}$ and $\mathcal{A}'$ are $\PGL$-equivalent.
\end{lemma}
\begin{proof}
	We can find a matrix $M\in\GL_{3}(\F_{2})$ such that $(\alpha_{1},\alpha_{2},\alpha_{3})M=(\alpha'_{1},\alpha'_{2},\alpha'_{3})$. The multiplicative group $\GL_{3}(\F_{2})$ can be generated by the matrices $M_{1}=\left(\begin{smallmatrix}
	1&0&0\\1&1&0\\0&0&1
	\end{smallmatrix}\right)$ and $M_{2}=\left(\begin{smallmatrix}
	0&1&0\\0&0&1\\1&0&0
	\end{smallmatrix}\right)$ and hence it is sufficient to prove the statement for $M=M_{1}$ and for $M=M_{2}$. Let $\beta_{1},\beta_{2},\beta_{3}$ be as in the construction presented in Theorem \ref{consq8}, so $\Tr(\alpha_{i}\beta_{j})=1$.
	\par We first look at $M=M_{1}$. In this case $(\alpha'_{1},\alpha'_{2},\alpha'_{3})=(\alpha_{1}+\alpha_{2},\alpha_{2},\alpha_{3})$. Then $(\beta'_{1},\beta'_{2},\beta'_{3})=(\beta_{1},\beta_{1}+\beta_{2},\beta_{3})$ fulfils $\Tr(\alpha'_{i}\beta'_{j})=\delta_{i,j}$. We know that the construction in Theorem \ref{consq8} does not depend on the choice of the coset leaders. So, when constructing the KM-arc using the triple $(\alpha'_{1},\alpha'_{2},\alpha'_{3})$ we may use $\beta'_{1},\beta'_{2},\beta'_{3}$ as coset leaders. The point set of $\mathcal{A}$ is given by $\mathcal{S}_{0}\cup\bigcup_{v\in\F^{3}_{2}}\mathcal{S}_{v}$ with
	\begin{align*}
	\mathcal{S}_{0}&=\{(0,1,x)\mid\forall i: \Tr(\alpha^{2}_{i}x)=0\}\text{ and }\\
	\mathcal{S}_{(\lambda_{1},\lambda_{2},\lambda_{3})}&=\left\{\left(1,\sum^{3}_{i=1}\lambda_{i}\alpha_{i},\sum^{3}_{i=1}f_{i}(\lambda_{1},\lambda_{2},\lambda_{3})\beta_{i}+s\right)\Bigg|\ s\in S\right\}\;,
	\end{align*}
	and the point set of $\mathcal{A}'$ is given by $\mathcal{S}'_{0}\cup\bigcup_{v\in\F^{3}_{2}}\mathcal{S}'_{v}$ with 	
	\begin{align*}
	\mathcal{S}'_{0}&=\{(0,1,x)\mid\forall i: \Tr(\alpha'^{2}_{i}x)=0\}\text{ and }\\
	\mathcal{S}'_{(\lambda_{1},\lambda_{2},\lambda_{3})}&=\left\{\left(1,\sum^{3}_{i=1}\lambda_{i}\alpha'_{i},\sum^{3}_{i=1}f_{i}(\lambda_{1},\lambda_{2},\lambda_{3})\beta'_{i}+s\right)\Bigg|\ s\in S\right\}\;.
	\end{align*}
	We know that $\left\langle\alpha^{2}_{1},\alpha^{2}_{2},\alpha^{2}_{3}\right\rangle=\left\langle\alpha'^{2}_{1},\alpha'^{2}_{2},\alpha'^{2}_{3}\right\rangle$. Keeping this in mind and using the above expressions for the $\alpha'_{i}$'s and $\beta'_{i}$'s, it can readily be checked that the collineation induced by the matrix $\left(\begin{smallmatrix}
	1&0&0\\\alpha_{2}+\alpha_{3}&1&0\\\beta_{1}+\beta_{3}&0&1
	\end{smallmatrix}\right)$ maps $\mathcal{A}$ onto $\mathcal{A'}$.
	\par Now we look at the case $M=M_{2}$, hence at $(\alpha'_{1},\alpha'_{2},\alpha'_{3})=(\alpha_{3},\alpha_{1},\alpha_{2})$. Then $(\beta'_{1},\beta'_{2},\beta'_{3})=(\beta_{3},\beta_{1},\beta_{2})$ fulfils $\Tr(\alpha'_{i}\beta'_{j})=\delta_{i,j}$. Again we can construct $\mathcal{A'}$ (note that $\mathcal{A}$ is as above). Here we can check that the collineation induced by the matrix $\left(\begin{smallmatrix}
	1&0&0\\\alpha_{1}+\alpha_{2}&1&0\\\beta_{2}+\beta_{3}&0&1
	\end{smallmatrix}\right)$ maps $\mathcal{A}$ onto $\mathcal{A'}$. 
\end{proof}

\begin{lemma}\label{projequiv2}
	Let $\mathcal{A}$ be a KM-arc of type $q/8$ in $\PG(2,q)$ obtained by the construction in Theorem \ref{consq8} using the admissible tuple $(\alpha_{1},\alpha_{2},\alpha_{3})$ and let $\mathcal{A'}$ be a KM-arc of type $q/8$ in $\PG(2,q)$ obtained by the construction in Theorem \ref{consq8} using $(k\alpha^{\varphi}_{1},k\alpha^{\varphi}_{2},k\alpha^{\varphi}_{3})$, with $k\in\F^{*}_{q}$ and $\varphi$ a field automorphism of $\F_{q}$. Then $\mathcal{A}$ and $\mathcal{A'}$ are $\PGammaL$-equivalent. 
\end{lemma}
\begin{proof}	
	Denote the set $\{s\mid \forall i:\Tr(\alpha_{i}s)=0\}$ by $S$ and the set $\{s\mid \forall i:\Tr(k\alpha^{\varphi}_{i}s)=0\}$ by $S'$. These sets are used in the construction of $\mathcal{A}$ and $\mathcal{A'}$, respectively. Let $\gamma$ be the collineation induced by the matrix $\left(\begin{smallmatrix} 	1&0&0\\0&k&0\\0&0&k^{-1} \end{smallmatrix}\right)$ and the field automorphism $\varphi$. Then $\mathcal{A}^{\gamma}=\mathcal{A'}$ since
	\begin{align*}
	\{k^{-1}s^{\varphi}\mid s\in S\}&=\{k^{-1}s^{\varphi}\mid\forall i:\Tr(\alpha_{i}s)=0\}=\{k^{-1}s^{\varphi}\mid\forall i:\Tr(\alpha^{\varphi}_{i}s^{\varphi})=0\}\\
	&=\{k^{-1}t\mid\forall i:\Tr(\alpha^{\varphi}_{i}t)=0\}=\{s\mid \forall i:\Tr(\alpha^{\varphi}_{i}ks)=0\}=S'\;,\\
	\{(0,k,k^{-1}x^{\varphi})\mid\forall i:\Tr(\alpha^{2}_{i}x)=0\}&=\{(0,1,k^{-2}x^{\varphi})\mid\forall i:\Tr(\alpha_{i}^{2}x)=0\}\\
	&=\{(0,1,k^{-2}x^{\varphi})\mid\forall i:\Tr((\alpha^{\varphi}_{i})^{2}x^{\varphi})=0\}\\
	&=\{(0,1,x)\mid\forall i:\Tr((k\alpha^{\varphi}_{i})^{2}x)=0\}
	\end{align*}
	and $\Tr((k\alpha^{\varphi}_{i})(k^{-1}\beta^{\varphi}_{j}))=\delta_{i,j}$. In both calculations we used that $\Tr(x^{\varphi})=\Tr(x)$ for the arbitrary field automorphism $\varphi$ and for any $x\in\F_{q}$.
\end{proof}

%There are more projective equivalences in this family of KM-arcs than the equivalences from Theorem \ref{projequiv1} and \ref{projequiv2}. We will see this below.

Combining the previous lemma with Lemma \ref{nottrans} yields that the constructed KM-arcs are not translation KM-arcs.

\begin{theorem}\label{nottranslation}
	If $\mathcal{A}$ is a KM-arc in $\PG(2,q)$ arising from the construction in Theorem \ref{consq8}, then $\mathcal{A}$ is not a translation KM-arc.
\end{theorem}
\begin{proof}
	For $q=16$ this result will follow from Theorem \ref{q=16}. Let $\mathcal{A}$ be a KM-arc in $\PG(2,q)$, $q\geq32$, constructed through Theorem \ref{consq8} using the admissible tuple $(\alpha_{1},\alpha_{2},\alpha_{3})$. We assume that $\mathcal{A}$ is a translation KM-arc. By \cite[Prop. 6.2]{km} (see also Theorem \ref{elationlinesecant}) the translation line must be a $q/8$-secant of $\mathcal{A}$.
	It follows from Theorem \ref{nottrans} that the translation line cannot be the elation line. So we assume that the translation line is a $q/8$-secant different from the elation line. Then, the subgroups $\{x\in\F_{q}\mid \forall i:\Tr(\alpha_{i}x)=0\}$ and $\{x\in\F_{q}\mid \forall i:\Tr(\alpha^{2}_{i}x)=0\}$ have to coincide, hence the subgroups $\langle\alpha_{1},\alpha_{2},\alpha_{3}\rangle$ and $\langle\alpha^{2}_{1},\alpha^{2}_{2},\alpha^{2}_{3}\rangle$ coincide. By Lemma \ref{projequiv1}, for every $k\in\F^{*}_{q}$, the admissible tuple $(k\alpha_{1},k\alpha_{2},k\alpha_{3})$, gives rise to a KM-arc $\mathcal{A'}$ $\PGammaL$-equivalent to $\mathcal{A}$, which is hence also a translation KM-arcs. As before, we find that the subgroups $\langle k\alpha_{1},k\alpha_{2},k\alpha_{3}\rangle$ and $\langle k^{2}\alpha^{2}_{1},k^{2}\alpha^{2}_{2},k^{2}\alpha^{2}_{3}\rangle$ coincide. It follows that for all $k\in\F^{*}_{q}$, the subgroups $\langle\alpha_{1},\alpha_{2},\alpha_{3}\rangle$ and $k\langle\alpha_{1},\alpha_{2},\alpha_{3}\rangle$ coincide, a contradiction, so the assumption is false.
\end{proof}

We now discuss the construction of Theorem \ref{consq8} for $q=16$ and $32$. For $q=16$ the construction of Theorem \ref{consq8} yields a hyperoval in $\PG(2,16)$. It is long known that up to isomorphism there are only two hyperovals in $\PG(2,16)$: the regular hyperoval and the Lunelli-Sce hyperoval (\cite{hall,okp}). The regular hyperoval has a stabiliser isomorphic to $\PGammaL(2,q)$ and hence has order $16320$, while the Lunelli-Sce has a stabiliser of order $144$ (see \cite{okp,pc}).

\begin{theorem}\label{q=16}
	For all admissible triples, the construction of Theorem \ref{consq8} for $q=16$ gives rise to the same hyperoval in $\PG(2,16)$ up to $\PGammaL$-equivalence; this hyperoval is the Lunelli-Sce hyperoval.
	%In $\PG(2,16)$ the construction of Theorem \ref{consq8} yields the Lunelli-Sce hyperoval for any $\alpha_{1},\alpha_{2},\alpha_{3}\in\F_{16}$ admitting the requirements. Consequently, all admissible parameter sets yield the same hyperoval up to projective equivalence.
\end{theorem}
\begin{proof}
	By Lemma \ref{projequiv1} we know that the projective equivalence class of the KM-arc does not depend on the choice of the parameters $\alpha_{1},\alpha_{2},\alpha_{3}$ (using the notation of Theorem \ref{consq8}) but only on the additive subgroup they generate. We know that $\F_{16}$ has precisely 15 additive subgroups of order 8. For any subgroup $S$ of order $8$ and any $k\in\F^{*}_{16}\setminus\{1\}$ clearly $kS\neq S$, so the 15 additive subgroups of order 8 can be written as $kT$ with $k\in\F^{*}_{16}$ and $T$ a fixed subgroup of order $8$. By Theorem \ref{projequiv2} we then know that all admissible triples give rise to the same hyperoval up to projective equivalence.  Using the GAP-package FinInG (\cite{fining}) we computed the stabiliser of one KM-arc in $\PG(2,16)$ constructed through Theorem \ref{consq8}. We found it to have size 144, hence the conclusion.
	%There are precisely 2520 triples $(\alpha_{1},\alpha_{2},\alpha_{3})\in\F^{3}_{16}$ such that $\left\langle\alpha_{1},\alpha_{2},\alpha_{3}\right\rangle_{2}$ is an additive subgroup of $\F_{16}$ of order $8$. For each of these triples we can construct the elation hyperoval as in Theorem \ref{consq8} and compute the subgroup of $\PGammaL(3,q)$ stabilising the hyperoval. The regular hyperoval has a stabiliser isomorphic to $\PGammaL(2,q)$ and hence has order $16320$, while the Lunelli-Sce has a stabiliser of order $144$ (see \cite{okp,pc}). Using the GAP-package FinInG (\cite{fining}) on a standard laptop we could compute all 2520 stabilisers in about three hours. All stabilisers have size 144, hence the conclusion.
	%We will use the notations from Theorem \ref{consq8} and its proof. The KM-arc $\mathcal{A}$ is in this case an elation hyperoval. It is either a regular hyperoval or else an Lunelli-Sce hyperoval.
	%It is known that the regular hyperoval $\mathcal{H}$ contains a point $K$ (commonly called the nucleus, but we will avoid this terminology here to avoid confusion with the $t$-nucleus) such that the automorphism group of $\mathcal{H}$ acts 3-transitive on the points of $\mathcal{H}\setminus\{K\}$ and fixes $K$.
	%\par We assume that $\mathcal{A}$ is a regular hyperoval.
\end{proof}

\begin{remark}\label{qover8in32}
	The KM-arcs of type $4$ in $\PG(2,32)$ have been classified up to projective equivalence in \cite[Result 2.14]{vdd}. There are $8$ equivalence classes. One of these classes was already described in \cite{kmm}. It is straightforward to check that only one of the 8 given KM-arcs is an elation KM-arc, the one whose affine points are given by $\{(1,f(z),z)\mid z\in\F_{32}\}$ with $f(z)=z^{24}+z^{20}+\alpha^{18}z^{18}+\alpha^5z^{16}+\alpha^2z^{12}+\alpha^{18}z^{10}+\alpha^{18}z^8+\alpha^{23}z^6+\alpha^{5}z^4+\alpha^{22}z^2+\alpha^{26}z$. The following result is immediate.  
\end{remark}

\begin{theorem}
    For all admissible triples the construction of Theorem \ref{consq8} for $q=32$ gives rise to the same elation KM-arc of type $4$ in $\PG(2,32)$ up to $\PGammaL$-equivalence.
\end{theorem}

Using the above results we can now give a computer free proof of this result.

\begin{proof}
	By Lemma \ref{projequiv1} we know that the projective equivalence class of the KM-arc only depends on the additive subgroup $\left\langle\alpha_{1},\alpha_{2},\alpha_{3}\right\rangle$ with $\alpha_{1},\alpha_{2},\alpha_{3}$ as in the statement of Theorem \ref{consq8}. It is immediate that $\F_{32}$ has $155$ additive subgroups of order 8. By Lemma \ref{projequiv2} we also know that for any subgroup $T$ the KM-arcs arising from $T$ and $kT^{\varphi}$, respectively, are $\PGammaL$-equivalent for any $k\in\F^{*}_{32}$ and any field automorphism $\varphi$ of $\F_{32}$. \par Assume that $\left\langle\alpha_{1},\alpha_{2},\alpha_{3}\right\rangle^{\varphi}_{2}=k\left\langle\alpha_{1},\alpha_{2},\alpha_{3}\right\rangle$, with $k\in\F^{*}_{32}$ and $\varphi\in\Aut(\F_{32})$. Then we can find a matrix $A\in\GL_{3}(\F_{2})$ such that $(\alpha^{\varphi}_{1},\alpha^{\varphi}_{2},\alpha^{\varphi}_{3})=k(\alpha_{1},\alpha_{2},\alpha_{3})A$. Applying this repetitively it follows that
	\[
	(\alpha_{1},\alpha_{2},\alpha_{3})=(\alpha^{\varphi^{5}}_{1},\alpha^{\varphi^{5}}_{2},\alpha^{\varphi^{5}}_{3})=k^{1+\varphi+\varphi^{2}+\varphi^{3}+\varphi^{4}}(\alpha_{1},\alpha_{2},\alpha_{3})A^{5}=(\alpha_{1},\alpha_{2},\alpha_{3})A^{5}\;
	\]
	since $\Aut(\F_{32})$ is a cyclic group of order $5$ which implies that $k^{1+\varphi+\varphi^{2}+\varphi^{3}+\varphi^{4}}=k^{31}=1$. As $|\GL_{3}(\F_{2})|=168$, the matrix $A$ cannot have order $5$, so $A$ is the identity matrix; here we also use that $\alpha_{1},\alpha_{2},\alpha_{3}$ are $\F_{2}$-independent. We find that $\alpha^{\varphi}_{i}=k\alpha_{i}$ for $i=1,2,3$. Consequently, $\frac{\alpha_{1}}{\alpha_{2}}$ is fixed by $\varphi$. As $\F_{2}$ is the only subfield of $\F_{32}$ and $\alpha_{1}\neq\alpha_{2}$, the field automorphism $\varphi$ must be trivial, and so also $k=1$.
	\par So, for a fixed additive subgroup $T$ of order $8$ in $\F_{32}$ all subgroups $kT^{\varphi}$, with $k\in\F^{*}_{32}$ and $\varphi\in\Aut(\F_{32})$, are different. For a fixed $T$ there are thus 31.5=155 subgroups of the form $kT^{\varphi}$. We conclude that all subgroups of order $8$ in $\F_{32}$ give rise to the same KM-arc of type $4$ up to projective equivalence.
\end{proof}

From \cite{vdd2} we also know that the stabiliser $G_{32}$ of this unique elation KM-arc $\mathcal{A}_{32}$ of type 4 is a group of order $16$. So, next to the four elations (including the identity) that $G_{32}$ contains by definition, there are other collineations stabilising $\mathcal{A}_{32}$; all of them fix only one point, the $4$-nucleus. It is clear that $\mathcal{A}_{32}$ is not translation.

Now we cover the larger values for $q$. First we recall a result from \cite{gw}. It learns us that the iterative process admitted by Construction \ref{gwc} (C) does not always construct `new' examples.

\begin{theorem}[{\cite[Remark 1]{gw}}]\label{gwremark}
	If $\mathcal{A}$ is the KM-arc in $\PG(2,q^{h})$ obtained from a hyperoval $\mathcal{H}$ in $\PG(2,q)$, $q$ even, through Construction \ref{gwc} (A) or (B), and $\mathcal{A'}$ is the KM-arc in $\PG(2,q^{hr})$ obtained from $\mathcal{A}$ through Construction 2 (C), then $\mathcal{A'}$ also arises from $\mathcal{H}$ through a direct application of Construction \ref{gwc} (A) or (B).
\end{theorem}

The proof of this theorem is straightforward; it can immediately be deduced from the descriptions in Construction \ref{gwc}. We now present an analogous theorem for the KM-arcs constructed in this section.

\begin{theorem}\label{iteratie} Let $\mathcal{A}_{0}$ be the KM-arc of type $q/8$ in $\PG(2,q)$, $q$ even, with $q/8$-nucleus $N(0,0,1)$ and elation line $X=0$ constructed from Theorem \ref{consq8} by the admissible tuple $(\alpha_{1},\alpha_{2},\alpha_{3})$ and let $\beta$ be a collineation that stabilises $N$. Let $\mathcal{A'}$ be a KM-arc of type $q^{h}/8$ in $\PG(2,q^{h})$ obtained from $\mathcal{A}_0^\beta$ through Construction \ref{gwc} (B) or (C). Then $\mathcal{A'}$ is $\PGammaL$-equivalent to a KM-arc in $\PG(2,q^h)$ obtained by the construction in Theorem \ref{consq8} using the admissible tuple $(\alpha_{1},\alpha_{2},\alpha_{3})$.
\end{theorem}
\begin{proof}
	We denote the trace function $\F_{q^{h}}\to\F_{2}$ by $\Tr_{q^{h}}$, the trace function  $\F_{q}\to\F_{2}$ by $\Tr_{q}$ and the trace function $\F_{q^{h}}\to\F_{q}$ by $\Tr_{q^{h},q}$. %Let $\mathcal{A}_{0}$ be the KM-arc of type $\frac{q}{8}$ with $\frac{q}{8}$-nucleus $N(0,0,1)$ and elation line $X=0$ directly constructed from Theorem \ref{consq8} and let $\beta$ be the collineation such that $\mathcal{A}^{\beta}_{0}=\mathcal{A}$. We may assume that $\beta$ stabilises $N$ as we apply Construction \ref{gwc} (B) or (C).
	\par %Let $\alpha_{1},\alpha_{2},\alpha_{3}\in\F_{q}$ be the parameters used for the construction of $\mathcal{A}_{0}$ as in Theorem \ref{consq8}. 
	We define $S=\{x\in\F_{q}\mid \forall i:\Tr_{q}(\alpha_{i}x)=0\}$. Then $\mathcal{A}_{0}$ is given by
	$\mathcal{S}_{0}\cup\bigcup_{v\in\F^{3}_{2}}\mathcal{S}_{v}$ with
	\[
	\mathcal{S}_{(\lambda_{1},\lambda_{2},\lambda_{3})}=\left\{\left(1,\sum^{3}_{i=1}\lambda_{i}\alpha_{i},\sum^{3}_{i=1}f_{i}(\lambda_{1},\lambda_{2},\lambda_{3})\beta_{i}+s\right)\Bigg|\  s\in S\right\}
	\]
	and $\mathcal{S}_{0}=\{(0,1,x)\mid\forall j: \Tr_{q}(\alpha^{2}_{j}x)=0\}$. Here, $\beta_{1},\beta_{2},\beta_{3}\in\F_{q}$ are such that $\Tr_{q}(\alpha_{i}\beta_{j})=\delta_{i,j}$. The collineation $\beta$ is defined by a matrix $C=\left(\begin{smallmatrix}
	a_{00}&a_{01}&0\\a_{10}&a_{11}&0\\a_{20}&a_{21}&1 
	\end{smallmatrix}\right)$ and an automorphism $\phi$ of $\F_{q}$.
	\par Let $k\in\F_{q^{h}}$ be such that $\Tr_{q^{h},q}(k)=1$; such an element can always be found. Now, we define $S'=\{x\in\F_{q^{h}}\mid \forall i:\Tr_{q^{h}}(k\alpha_{i}x)=0\}$. On the one hand, for any element $x\in\F_{q}\subseteq\F_{q^{h}}$ we know that $\Tr_{q^{h}}(k\alpha_{i}x)=\Tr_{q}(\alpha_{i}x\Tr_{q^{h},q}(k))=\Tr_{q}(\alpha_{i}x)$. Hence, $S'\cap\F_{q}=S$. On the other hand, for any element $x\in\F_{q^{h}}$ with $\Tr_{q^{h},q}(kx)=0$ we know that $\Tr_{q^{h}}(k\alpha_{i}x)=\Tr_{q}(\alpha_{i}\Tr_{q^{h},q}(kx))=0$. Moreover, if $x\in\F_{q}\subseteq\F_{q^{h}}$ admits $\Tr_{q^{h},q}(kx)=0$ then $x=0$. So the set $I=\{x\in \F_{q^h}\mid\Tr_{q^{h},q}(kx)=0\}$ is a direct complement of $\F_{q}$ in $\F_{q^{h}}$ such that $S'=\langle S,I\rangle$. We can also find an automorphism $\phi'$ of $\F_{q^{h}}$ of which $\phi$ is the restriction to $\F_{q}$. Then $I^{\phi'}$ is also a direct complement of $\F_{q}$ in $\F_{q^{h}}$.
	\par By Remark \ref{directcomplements} we may use $I^{\phi'}$ in the construction of $\mathcal{A'}$ without loss of generality. The KM-arc $\mathcal{A'}$ is then given by $\mathcal{S}'_{0}\cup\bigcup_{v\in\F^{3}_{2}}\mathcal{S}'_{v}$ with $\mathcal{S}'_{0}=\{(0,1,x^{\phi}+i^{\phi'})C^{t}\mid\forall j: \Tr_{q}(\alpha^{2}_{j}x)=0, i\in I\}$ and% $\mathcal{S}'_{(\lambda_{1},\lambda_{2},\lambda_{3})}=$
	\begin{align*}
	\mathcal{S}'_{(\lambda_{1},\lambda_{2},\lambda_{3})}&=\left\{\left(1,\sum^{3}_{i=1}\lambda_{i}\alpha^{\phi}_{i},\sum^{3}_{i=1}f_{i}(\lambda_{1},\lambda_{2},\lambda_{3})\beta^{\phi}_{i}+s^{\phi}+i^{\phi'}\right)C^{t}\Bigg|\ s\in S, i\in I\right\}\\
	&=\left\{\left(1,\sum^{3}_{i=1}\lambda_{i}\alpha^{\phi}_{i},\sum^{3}_{i=1}f_{i}(\lambda_{1},\lambda_{2},\lambda_{3})\beta^{\phi}_{i}+s^{\phi'}\right)C^{t}\Bigg|\ s\in S'\right\}\;.
	\end{align*}
	We define the KM-arc $\mathcal{A}''$ in $\PG(2,q^{h})$ using the parameters $k\alpha_{1},k\alpha_{2},k\alpha_{3}$. Its point set is given by $\mathcal{S}''_{0}\cup\bigcup_{v\in\F^{3}_{2}}\mathcal{S}''_{v}$ with $\mathcal{S}''_{0}=\{(0,1,x)\mid\forall j: \Tr_{q^{h}}(k^{2}\alpha^{2}_{j}x)=0\}$ and
	\[
	\mathcal{S}''_{(\lambda_{1},\lambda_{2},\lambda_{3})}=\left\{\left(1,k\sum^{3}_{i=1}\lambda_{i}\alpha_{i},\sum^{3}_{i=1}f_{i}(\lambda_{1},\lambda_{2},\lambda_{3})\beta_{i}+s\right)\Bigg|\ s\in S'\right\}\;.
	\]
	Note that $\beta_{1},\beta_{2},\beta_{3}\in\F_{q}\subseteq\F_{q^{h}}$ fulfil $\Tr_{q^{h}}((k\alpha_{i})\beta_{j})=\Tr_{q}(\alpha_{i}\beta_{j})=\delta_{i,j}$. Let $\gamma$ be the collineation induced by the $\F_{q^{h}}$-automorphism $\phi'$ and the matrix $C'=C\left(\begin{smallmatrix}
	1&0&0\\0&(k^{\phi'})^{-1}&0\\0&0&1 
	\end{smallmatrix}\right)$ where we interpret $C$ over $\F_{q^{h}}$. It is immediate that $\left(\mathcal{S}''_{(\lambda_{1},\lambda_{2},\lambda_{3})}\right)^{\gamma}=\mathcal{S}'_{(\lambda_{1},\lambda_{2},\lambda_{3})}$ for all $(\lambda_{1},\lambda_{2},\lambda_{3})\in\F^{3}_{2}$. Furthermore,
	\begin{align*}
	\left(\mathcal{S}''_{0}\right)^{\gamma}&=\{(0,(k^{\phi'})^{-1},x^{\phi'})C^{t}\mid\forall j: \Tr_{q^{h}}(k^{2}\alpha^{2}_{j}x)=0\}\\
	&=\{(0,1,(kx)^{\phi'})C^{t}\mid\forall j: \Tr_{q^{h}}(k^{2}\alpha^{2}_{j}x)=0\}\\
	&=\{(0,1,y^{\phi'})C^{t}\mid\forall j: \Tr_{q^{h}}(k\alpha^{2}_{j}y)=0\}\;.\\
	\end{align*}
	So, $\left(\mathcal{S}''_{0}\right)^{\gamma}=\mathcal{S}'_{0}$ iff $\Tr_{q^{h}}(k\alpha^{2}_{j}(x+i))=0$ for all $i\in I$ and all $x\in\F_{q}$ such that $\Tr_{q}(\alpha^{2}_{j}x)=0$. We find
	\[
	\Tr_{q^{h}}(k\alpha^{2}_{j}(x+i))=\Tr_{q^{h}}(k\alpha^{2}_{j}x)+\Tr_{q^{h}}(k\alpha^{2}_{j}i)=\Tr_{q}(\alpha^{2}_{j}x\Tr_{q^{h},q}(k))+\Tr_{q}(\alpha^{2}_{j}\Tr_{q^{h},q}(ki))=0
	\]
	by the definition of $k$ and the definition of $I$. We conclude that $\mathcal{A'}=(\mathcal{A}'')^{\gamma}$. This proves the theorem since the tuples $(\alpha_{1},\alpha_{2},\alpha_{3})$ and $(k\alpha_{1},k\alpha_{2},k\alpha_{3})$ give rise to $\PGammaL$-equivalent KM-arcs by Lemma \ref{projequiv2}.
\end{proof}

We now discuss in detail the result of Theorem \ref{consq8} and Corollary \ref{q8exist}.

\begin{remark}\label{q8detail}
	In $\PG(2,q)$, $q=2^h$, with $3\mid h$, KM-arcs of type $q/8$ were known to exist through Constructions \ref{kmc} and \ref{gwc} (A). However, since all $o$-polynomials in $\F_{8}$ give rise to a translation hyperoval (see \cite{seg}), all these KM-arcs are translation KM-arcs. By Theorem \ref{gwremark} all KM-arcs of type $q/8$ that are constructed through applying Construction \ref{gwc} (C) on the previous ones, are also translation KM-arcs.
	\par In $\PG(2,q)$, $q=2^h$, with $4\mid h$, KM-arcs of type $q/8$ were known to exist through Construction \ref{gwc} (B). By Theorem \ref{gwremark} all KM-arcs that arise through Constructions \ref{gwc} (B) and (C) arise from a hyperoval in $\PG(2,16)$. They are all elation KM-arcs.
	\par In $\PG(2,q)$, $q=2^h$, with $5\mid h$, KM-arcs of type $q/8$ were known to exist through Remark \ref{qover8in32} and Construction \ref{gwc} (C). By Lemmas \ref{elationals} and \ref{elationals2} this family contains both elation and non-elation KM-arcs.
	\par By Theorem \ref{nottranslation} we know that the KM-arcs constructed through Theorem \ref{consq8} are not translation KM-arcs. We now elaborate on Corollary \ref{q8exist}. For the discussion of the existence results of KM-arcs of type $q/8$ in $\PG(2,q)$, $q=2^h$, the residue class of $h$ modulo $60$ is what matters. If $h\neq 0\pmod{m}$ for $m=3,4,5$ (there are 24 out of 60 residue classes in this case), then the existence of KM-arcs of type $q/8$ in $\PG(2,q)$ was previously not known. If $h$ is divisible by $3$, but not by $4$ or $5$ (there are 12 residue classes in this case), then the existence of KM-arcs of type $q/8$ in $\PG(2,q)$ was previously known, but all known examples are translation KM-arcs and hence different from the examples we introduced in Theorem \ref{consq8} as they are not translation KM-arcs. If $h$ is divisible by $4$ or by $5$ (there are 24 residue classes in this case), then the set of KM-arcs constructed in Theorem \ref{consq8} does not necessarily contain previously unknown examples. E.g. for $h=4,5$ this construction provides no new KM-arcs.
\end{remark}

\section{A new family of elation KM-arcs of type \texorpdfstring{$q/16$}{q/16}}\label{q/16}

In this section we first present the construction of a family of KM-arcs of type $q/16$ in $\PG(2,q)$, based on the idea underlying the construction of KM-arcs of type $q/8$ in Theorem \ref{consq8}. Afterwards we will discuss this family of KM-arcs. We start with a small technical lemma.

\begin{lemma}\label{voorfactor}
	Let $\alpha_{1},\alpha_{2},\alpha_{3},\alpha_{4}\in\F^{*}_{q}$ be $\F_{2}$-independent. If $\frac{\alpha^{2}_{i}}{\alpha_{4}}\in\left\langle\alpha_{1},\alpha_{2},\alpha_{3},\alpha_{4}\right\rangle$ for $i=1,2,3$, then we can find an $\alpha\in\left\langle\alpha_{1},\alpha_{2},\alpha_{3},\alpha_{4}\right\rangle$ such that $\{\alpha_{1}(\alpha_{1}+\alpha_{4}),\alpha_{2}(\alpha_{2}+\alpha_{4}),\alpha_{3}(\alpha_{3}+\alpha_{4}),\alpha_{4}\alpha\}$ is an $\F_{2}$-independent set.
\end{lemma}
\begin{proof}
	If the triple $(\mu_{1},\mu_{2},\mu_{3})\in\F_{2}^{3*}$ admits $\sum_{i=1}^{3}\mu_{i}\alpha_{i}(\alpha_{i}+\alpha_{4})=0$, then $\left(\sum_{i=1}^{3}\mu_{i}\alpha_{i}\right)\left(\alpha_{4}+\sum_{i=1}^{3}\mu_{i}\alpha_{i}\right)=0$, contradicting that $\{\alpha_{1},\alpha_{2},\alpha_{3},\alpha_{4}\}$ is an $\F_{2}$-independent set, so $\{\alpha_{1}(\alpha_{1}+\alpha_{4}),\alpha_{2}(\alpha_{2}+\alpha_{4}),\alpha_{3}(\alpha_{3}+\alpha_{4})\}$ is an $\F_{2}$-independent set.
	\par Since $\frac{\alpha^{2}_{i}}{\alpha_{4}}\in\left\langle\alpha_{1},\alpha_{2},\alpha_{3},\alpha_{4}\right\rangle$ for $i=1,2,3$, there exist $a_{i,j}\in\F_{2}$ such that $\alpha^{2}_{i}=\alpha_{4}\left(\sum_{j=1}^{4}a_{i,j}\alpha_{j}\right)$. Let $(b_{1},b_{2},b_{3},b_{4})$ be a vector in $\F_{2}^{4}$  which is not contained in the hyperplane
	\[
	\left\langle(a_{1,1}+1,a_{1,2},a_{1,3},a_{1,4}),(a_{2,1},a_{2,2}+1,a_{2,3},a_{2,4}),(a_{3,1},a_{3,2},a_{3,3}+1,a_{3,4})\right\rangle\;.
	\]
	We then know that for $\alpha=\sum_{i=1}^{4}b_{i}\alpha_{i}$, the set $\{\alpha_{1}(\alpha_{1}+\alpha_{4}),\alpha_{2}(\alpha_{2}+\alpha_{4}),\alpha_{3}(\alpha_{3}+\alpha_{4}),\alpha_{4}\alpha\}$ is $\F_{2}$-independent.
\end{proof}

Now we present the construction. By the previous lemma we know that the existence of an $\alpha$ satisfying the condition in the theorem is guaranteed.

\begin{theorem}\label{cons}
	Let $q=2^h$, $h>5$, let $\alpha_{1},\alpha_{2},\alpha_{3},\alpha_{4}\in\F^{*}_{q}$ be $\F_{2}$-independent and define $S=\{x\in\F_{q}\mid \forall i:\Tr(\alpha_{i}x)=0\}$. Assume that $\frac{\alpha^{2}_{i}}{\alpha_{4}}\in\left\langle\alpha_{1},\alpha_{2},\alpha_{3},\alpha_{4}\right\rangle$ for $i=1,2,3$, and let $\alpha\in\left\langle\alpha_{1},\alpha_{2},\alpha_{3},\alpha_{4}\right\rangle$ be such that $\{\alpha_{1}(\alpha_{1}+\alpha_{4}),\alpha_{2}(\alpha_{2}+\alpha_{4}),\alpha_{3}(\alpha_{3}+\alpha_{4}),\alpha_{4}\alpha\}$ is an $\F_{2}$-independent set. Let  $\beta_{1},\beta_{2},\beta_{3}\in\F^{*}_{q}$ be such that $\Tr(\alpha_{i}\beta_{j})=\delta_{i,j}$ for $i=1,\dots,4$ and $j=1,2,3$, and let $f_{1},f_{2},f_{3}$ be as in Theorem \ref{consq8}.
	\par For any $\overline{\lambda}=(\lambda_{1},\dots,\lambda_{4})\in\F^{4}_{2}$ we define
	\[
	\mathcal{S}_{\overline{\lambda}}=\left\{\left(1,\sum^{4}_{i=1}\lambda_{i}\alpha_{i},\sum^{3}_{i=1}f_{i}(\lambda_{1},\lambda_{2},\lambda_{3})\beta_{i}+s\right)\Bigg|\ s\in S\right\}\;.
	\]
	We also define $\mathcal{S}_{0}=\{(0,1,x)\mid \Tr(\alpha_{i}(\alpha_{i}+\alpha_{4})x)=0,\: i=1,2,3\ \wedge\ \Tr(\alpha_{4}\alpha x)=1\}$. The point set $\mathcal{A}=\mathcal{S}_{0}\cup\bigcup_{v\in\F^{4}_{2}}\mathcal{S}_{v}$ is an elation KM-arc of type $q/16$ in $\PG(2,q)$ with elation line $X=0$ and $q/16$-nucleus $(0,0,1)$.
\end{theorem}
\begin{proof}
	We follow the approach from the proof of Theorem \ref{consq8}. We know that $S$ is a subgroup of $\F_{q},+$ containing $q/16$ elements. The existence of elements $\beta_{1},\beta_{2},\beta_{3}$ is guaranteed as they are coset leaders of cosets of $S$ (note that not all cosets of $S$ are involved).
	\par It is immediate that the points of $\mathcal{S}_{\overline{\lambda}}$, with $\overline{\lambda}=(\lambda_{1},\lambda_{2},\lambda_{3},\lambda_{4})\in\F^{4}_{2}$, are on the line $\ell_{\overline{\lambda}}$ with equation $(\sum^{4}_{i=1}\lambda_{i}\alpha_{i})X+Y=0$ and that the points of $\mathcal{S}_{0}$ are on the line $\ell_{\infty}$ with equation $X=0$. Hence, all lines through $N(0,0,1)$ either contain $q/16$ or $0$ points of $\mathcal{A}$.
	\par Now we check that three points on different $q/16$-secants are not collinear. First we assume that $\ell_{\infty}$ is not among these three $q/16$-secants. Then the three points can be described as
	\[
	\left(1,\sum^{4}_{i=1}\lambda_{i}\alpha_{i},\sum^{3}_{i=1}f_{i}(\widetilde{\lambda})\beta_{i}+s\right), \left(1,\sum^{4}_{i=1}\lambda'_{i}\alpha_{i},\sum^{3}_{i=1}f_{i}(\widetilde{\lambda}')\beta_{i}+s'\right)\text{ and } \left(1,\sum^{4}_{i=1}\lambda''_{i}\alpha_{i},\sum^{3}_{i=1}f_{i}(\widetilde{\lambda}'')\beta_{i}+s''\right)
	\] 
	with $\overline{\lambda}=(\lambda_{1},\lambda_{2},\lambda_{3},\lambda_{4})$, $\overline{\lambda}'=(\lambda'_{1},\lambda'_{2},\lambda'_{3},\lambda'_{4})$ and $\overline{\lambda}''=(\lambda''_{1},\lambda''_{2},\lambda''_{3},\lambda''_{4})$ three pairwise different vectors in $\F^{4}_{2}$, and $\widetilde{\lambda}=(\lambda_{1},\lambda_{2},\lambda_{3})$, $\widetilde{\lambda}'=(\lambda'_{1},\lambda'_{2},\lambda'_{3})$ and $\widetilde{\lambda}''=(\lambda''_{1},\lambda''_{2},\lambda''_{3})$. We find that
	\begin{align*}
	\Delta&=\begin{vmatrix}
	1 & \sum^{4}_{i=1}\lambda_{i}\alpha_{i} & \sum^{3}_{i=1}f_{i}(\widetilde{\lambda})\beta_{i}+s\\
	1 & \sum^{4}_{i=1}\lambda'_{i}\alpha_{i} & \sum^{3}_{i=1}f_{i}(\widetilde{\lambda}')\beta_{i}+s'\\
	1 & \sum^{4}_{i=1}\lambda''_{i}\alpha_{i} & \sum^{3}_{i=1}f_{i}(\widetilde{\lambda}'')\beta_{i}+s''
	\end{vmatrix}\\
	&=\sum_{cyc}\left(\sum_{i=1}^{4}\sum_{j=1}^{3}(\lambda_{i}f_{j}(\widetilde{\lambda}')+\lambda'_{i}f_{j}(\widetilde{\lambda}))\alpha_{i}\beta_{j}+s\sum^{4}_{i=1}\lambda'_{i}\alpha_{i}+s'\sum^{4}_{i=1}\lambda_{i}\alpha_{i}\right)
	\end{align*}
	where the cyclic sum is taken over $(\overline{\lambda},\overline{\lambda}',\overline{\lambda}'')$ and the corresponding $(\widetilde{\lambda},\widetilde{\lambda}',\widetilde{\lambda}'')$ and $(s,s',s'')$.  We calculate the trace of both sides of this equation. Considering that $\Tr(\alpha_{i}t)=0$ for all $t\in S$ and $i=1,\dots,4$, that $\Tr(\alpha_{i}\beta_{j})=\delta_{i,j}$ and that the trace function is $\F_{2}$-linear, we find (completely analogous to \eqref{m3}) that
	\begin{align*}
	\Tr(\Delta)%&=\sum_{cyc}\left(\sum_{i,j=1}^{3}(\lambda_{i}f_{i}(\widetilde{\lambda}')+\lambda'_{i}f_{i}(\widetilde{\lambda}))\right)\\
	%&=\sum_{cyc}\left(\lambda_{1}(\lambda'_{2}+\lambda'_{3}+\lambda'_{2}\lambda'_{3})+\lambda'_{1}(\lambda_{2}+\lambda_{3}+\lambda_{2}\lambda_{3})+\lambda_{2}(\lambda'_{3}+\lambda'_{1}\lambda'_{3})+\lambda'_{2}(\lambda_{3}+\lambda_{1}\lambda_{3})\right.\\
	%&\qquad\qquad\left.+\lambda_{3}\lambda'_{1}\lambda'_{2}+\lambda'_{3}\lambda_{1}\lambda_{2}\right)\\
	%&=\sum_{cyc}\left(M^{3}_{2}(\widetilde{\lambda},\widetilde{\lambda}')+\sum_{i=1}^{3}(\lambda_{i}+1)+\sum_{i=1}^{3}(\lambda'_{i}+1)+1\right)\\
	%&=M^{3}_{2}(\widetilde{\lambda},\widetilde{\lambda}')+M^{3}_{2}(\widetilde{\lambda}',\widetilde{\lambda}'')+M^{3}_{2}(\widetilde{\lambda}'',\widetilde{\lambda})\\
	&=M^{3}_{3}(\widetilde{\lambda},\widetilde{\lambda}',\widetilde{\lambda}'')\;.
	\end{align*}
	It follows that $\Delta\neq0$ if $\widetilde{\lambda}$, $\widetilde{\lambda}'$ and $\widetilde{\lambda}''$ are three pairwise disjoint vectors. Hence, in this case the three points are not collinear. Now we look at the case in which $\widetilde{\lambda}$, $\widetilde{\lambda}'$ and $\widetilde{\lambda}''$ are not three pairwise disjoint vectors. Without loss of generality we can assume that $\widetilde{\lambda}'=\widetilde{\lambda}''$. Since $\overline{\lambda}'\neq\overline{\lambda}''$, we know that also $\lambda'_{4}=\lambda''_{4}+1$. In this case
	\begin{align*}
	\Delta=\alpha_{4}\sum_{j=1}^{3}(f_{j}(\widetilde{\lambda})+f_{j}(\widetilde{\lambda}'))\beta_{j}+(s+s')\alpha_{4}+(s'+s'')\sum_{i=1}^{4}(\lambda_{i}+\lambda'_{i})\alpha_{i}
	\end{align*}
	Since $\overline{\lambda}\neq\overline{\lambda}'$ by assumption, we know that $\sum_{i=1}^{4}(\lambda_{i}+\lambda'_{i})\alpha_{i}\neq 0$. So now we compute the trace of the following nonzero multiple of $\Delta$:
	\begin{align*}
	\Tr\left(\frac{\sum_{i=1}^{4}(\lambda_{i}+\lambda'_{i})\alpha_{i}}{\alpha_{4}}\Delta\right)&=\sum_{i=1}^{4}\sum_{j=1}^{3}(\lambda_{i}+\lambda'_{i})(f_{j}(\widetilde{\lambda})+f_{j}(\widetilde{\lambda}'))\Tr(\alpha_{i}\beta_{j})\\
	&\qquad+\sum_{i=1}^{4}(\lambda_{i}+\lambda'_{i})\Tr((s+s')\alpha_{i})+\Tr\left((s'+s'')\frac{\sum_{i=1}^{4}(\lambda_{i}+\lambda'_{i})\alpha^{2}_{i}}{\alpha_{4}}\right)\\
	&=\sum_{i=1}^{3}(\lambda_{i}+\lambda'_{i})(f_{i}(\widetilde{\lambda})+f_{i}(\widetilde{\lambda}'))\\
	&=(\lambda_{1}+\lambda'_{1})(\lambda_{1}+\lambda_{2}+\lambda_{3}+\lambda_{2}\lambda_{3}+\lambda'_{1}+\lambda'_{2}+\lambda'_{3}+\lambda'_{2}\lambda'_{3})\\&\qquad+(\lambda_{2}+\lambda'_{2})(\lambda_{2}+\lambda_{3}+\lambda_{1}\lambda_{3}+\lambda'_{2}+\lambda'_{3}+\lambda'_{1}\lambda'_{3})\\
	&\qquad+(\lambda_{3}+\lambda'_{3})(\lambda_{3}+\lambda_{1}\lambda_{2}+\lambda'_{3}+\lambda'_{1}\lambda'_{2})\\
	&=M^{3}_{2}(\widetilde{\lambda},\widetilde{\lambda}')\;.
	\end{align*}
	In the second step we used that $\frac{\alpha^{2}_{i}}{\alpha_{4}}\in\left\langle\alpha_{1},\alpha_{2},\alpha_{3},\alpha_{4}\right\rangle$, hence that $\Tr\left(s\frac{\alpha^{2}_{i}}{\alpha_{4}}\right)=0$ for all $s\in S$. In the final step we used that all elements of $\F_{2}$ equal their square. As $\overline{\lambda}$ differs from both $\overline{\lambda}'$ and $\overline{\lambda}''$, which only differ on the final entry, the vector $\widetilde{\lambda}$ has to be different from $\widetilde{\lambda}'$. It follows that $\Delta\neq0$, hence, also in this case the three points are not collinear.
	\par Now we assume that $\ell_{\infty}$ is among the three $q/16$-secants. Then, the three points can be described as
	\[
	\left(1,\sum^{4}_{i=1}\lambda_{i}\alpha_{i},\sum^{3}_{i=1}f_{i}(\widetilde{\lambda})\beta_{i}+s\right), \left(1,\sum^{4}_{i=1}\lambda'_{i}\alpha_{i},\sum^{3}_{i=1}f_{i}(\widetilde{\lambda}')\beta_{i}+s'\right)\text{ and } \left(0,1,t\right)
	\] 
	with $\overline{\lambda}=(\lambda_{1},\lambda_{2},\lambda_{3},\lambda_{4})$ and $\overline{\lambda}'=(\lambda'_{1},\lambda'_{2},\lambda'_{3},\lambda'_{4})$ two different vectors in $\F^{3}_{2}$, $\widetilde{\lambda}=(\lambda_{1},\lambda_{2},\lambda_{3})$ and $\widetilde{\lambda}'=(\lambda'_{1},\lambda'_{2},\lambda'_{3})$, and $t$ such that $\Tr(\alpha_{i}(\alpha_{i}+\alpha_{4})t)=0$ for $i=1,2,3$ and $\Tr(\alpha_{4}\alpha t)=1$. We find that
	\begin{align*}
	\Delta'&=\begin{vmatrix}
	0 & 1 & t\\
	1 & \sum^{4}_{i=1}\lambda_{i}\alpha_{i} & \sum^{3}_{i=1}f_{i}(\widetilde{\lambda})\beta_{i}+s\\
	1 & \sum^{4}_{i=1}\lambda'_{i}\alpha_{i} & \sum^{3}_{i=1}f_{i}(\widetilde{\lambda}')\beta_{i}+s'
	\end{vmatrix}=t\sum^{4}_{i=1}(\lambda_{i}+\lambda'_{i})\alpha_{i}+\sum_{i=1}^{3}(f_{i}(\widetilde{\lambda})+f_{i}(\widetilde{\lambda}'))\beta_{i}+s+s'\;.
	\end{align*}
	We know that $\overline{\lambda}+\overline{\lambda}'\neq 0$ and hence $\sum_{i=1}^{4}(\lambda_{i}+\lambda'_{i})\alpha_{i}\neq 0$. We distinguish between two cases. First we assume that $\widetilde{\lambda}\neq\widetilde{\lambda}'$. We compute the trace of a nonzero multiple of $\Delta'$:
	\begin{align*}
	\Tr\left(\left(\alpha_{4}+\sum_{i=1}^{4}(\lambda_{i}+\lambda'_{i})\alpha_{i}\right)\Delta'\right)&=\sum^{4}_{i=1}(\lambda_{i}+\lambda'_{i})\Tr\left(t\alpha_{i}(\alpha_{i}+\alpha_{4})\right)+\sum_{j=1}^{3}(f_{j}(\widetilde{\lambda})+f_{j}(\widetilde{\lambda}'))\Tr(\alpha_{4}\beta_{j})\\
	&\qquad\qquad+\sum_{i=1}^{4}\sum_{j=1}^{3}(\lambda_{i}+\lambda'_{i})(f_{j}(\widetilde{\lambda})+f_{j}(\widetilde{\lambda}'))\Tr(\alpha_{i}\beta_{j})\\
	&\qquad\qquad+\sum_{i=1}^{4}(\lambda_{i}+\lambda'_{i})\Tr\left((s+s')\alpha_{i}\right)+\Tr\left(\alpha_{4}(s+s')\right)\\
	&=\sum_{i=1}^{3}(\lambda_{i}+\lambda'_{i})(f_{i}(\widetilde{\lambda})+f_{i}(\widetilde{\lambda}'))\\
	&=M^{3}_{2}(\widetilde{\lambda},\widetilde{\lambda}')\;.
	\end{align*}
	In the penultimate step we used that $\Tr(\alpha_{i}(\alpha_{i}+\alpha_{4})t)=0$ for $i=1,2,3$ (and trivially also for $i=4$) and that $\Tr(\alpha_{i}\beta_{i})=\delta_{i,j}$. In the final step we used the calculations in \eqref{veellambdas}. We find that $\Delta\neq0$, hence the three points are not collinear.
	\par If $\widetilde{\lambda}=\widetilde{\lambda}'$, then $\overline{\lambda}+\overline{\lambda}'=\alpha_{4}$. In this case $\Delta'=t\alpha_{4}+s+s'$. We know that
	\[
	\Tr(\alpha\Delta')=\Tr(\alpha(t\alpha_{4}+s+s'))=1
	\]
	because $\Tr(\alpha_{4}\alpha t)=1$ and $\alpha\in\left\langle\alpha_{1},\alpha_{2},\alpha_{3},\alpha_{4}\right\rangle$. Again we find that $\Delta\neq0$, hence the three points are not collinear
	\par We conclude that all lines not through $N$ contain at most two points of $\mathcal{A}$. For any point $P\in\mathcal{A}$ there are $q$ points of $\mathcal{A}$ not on the $q/16$-secant $\ell_{P}=\left\langle P,N\right\rangle$ so all $q$ lines through $P$ different from $\ell_{P}$ contain precisely two points of $\mathcal{A}$. Consequently, all lines of $\PG(2,q)$ contain 0, 2 or $q/16$ points of $\mathcal{A}$. So, $\mathcal{A}$ is a KM-arc of type $q/16$. From its definition and Lemma \ref{Sinvoeren} it follows immediately that $\mathcal{A}$ is an elation KM-arc with elation line $X=0$.
\end{proof}

\begin{remark}
	It is clear from the proof of Lemma \ref{voorfactor} that there are 8 possible choices for the $\alpha$ used in the construction of Theorem \ref{cons}. We can show that the construction is independent of the chosen $\alpha$. Assume that $\alpha'$ and $\alpha''$ are such that both  $\{\alpha_{1}(\alpha_{1}+\alpha_{4}),\alpha_{2}(\alpha_{2}+\alpha_{4}),\alpha_{3}(\alpha_{3}+\alpha_{4}),\alpha_{4}\alpha'\}$ and $\{\alpha_{1}(\alpha_{1}+\alpha_{4}),\alpha_{2}(\alpha_{2}+\alpha_{4}),\alpha_{3}(\alpha_{3}+\alpha_{4}),\alpha_{4}\alpha''\}$ are $\F_{2}$-independent sets. We know that $\alpha_{4}(\alpha'+\alpha'')$ is contained in $\left\langle\alpha_{1}(\alpha_{1}+\alpha_{4}),\alpha_{2}(\alpha_{2}+\alpha_{4}),\alpha_{3}(\alpha_{3}+\alpha_{4})\right\rangle$. Consequently,
	\begin{align*}
	&\{(0,1,x)\mid \Tr(\alpha_{i}(\alpha_{i}+\alpha_{4})x)=0,\: i=1,2,3\ \wedge\ \Tr(\alpha_{4}\alpha' x)=1\}\\=\:&\{(0,1,x)\mid \Tr(\alpha_{i}(\alpha_{i}+\alpha_{4})x)=0,\: i=1,2,3\ \wedge\ \Tr(\alpha_{4}\alpha'' x)=1\}\;,
	\end{align*}
	which proves our claim.
\end{remark}

The following result follows immediate from the definition of the KM-arcs of type $q/16$ constructed in Theorem \ref{cons}.

\begin{theorem}
	The $q/16$-secants of the elation KM-arc of type $q/16$ constructed in Theorem \ref{cons} define an $\F_2$-linear pencil.
\end{theorem}

This result is similar to Theorem \ref{linpenq/8} where we have showed that the same holds for KM-arcs constructed in Theorem \ref{consq8}. It would be interesting to know whether Theorem \ref{elationLinPencil} is valid for all KM-arcs of type $q/16$ (see also Remark \ref{remlinearpencil}).

\begin{lemma}\label{nottrans16}
	The KM-arc constructed in Theorem \ref{cons} is not a translation KM-arc with the elation line as translation line.
\end{lemma}
\begin{proof}
	Analogous to the proof of Lemma \ref{nottrans}.
\end{proof}

We look at the elations stabilising a KM-arc constructed through Theorem \ref{cons}. We know by Lemma \ref{Sinvoeren} that all elation KM-arcs of type $q/16$ in $\PG(2,q)$ admit a group of elations of size $q/16$. We will now prove that the KM-arcs constructed through Theorem \ref{cons} are stabilised by a larger group of elations.

\begin{theorem}\label{elationsubgroup}
	A KM-arc $\mathcal{A}$ of type $q/16$ in $\PG(2,q)$ constructed though Theorem \ref{cons} admits a group of elations of size $q/8$.
\end{theorem}
\begin{proof}
	We assume $\mathcal{A}$ is the point set given in the statement of Theorem \ref{cons} with $\alpha_{1},\dots,\alpha_{4}$ and $S$ as described there. It can readily be checked that all elations in $E=\left\{\left(\begin{smallmatrix}1&0&0\\k \alpha_{4}&1&0\\s&0&1\end{smallmatrix}\right)\mid s\in S, k\in\F_{2}\right\}$ fix $\mathcal{A}$. It is also immediate that $E$ is a group of elations with axis $X=0$ and that $E$ has size $q/8$.
\end{proof}

%\begin{corollary}
	%For any $q$ the family of KM-arcs in $\PG(2,q)$ constructed in Theorem \ref{cons} contains elation KM-arcs which are not translation KM-arcs.
%\end{corollary}
%\begin{proof}
	%Similar to the proof of Corollary \ref{notalltranslation}. We use the notations of Theorem \ref{cons}. By the condition $\frac{\alpha^{2}_{i}}{\alpha_{4}}\in\left\langle\alpha_{1},\alpha_{2},\alpha_{3},\alpha_{4}\right\rangle_{2}$ for $i=1,2,3$ we know that the subgroup $\left\langle\alpha_{1}(\alpha_{1}+\alpha_{4}),\alpha_{2}(\alpha_{2}+\alpha_{4}),\alpha_{3}(\alpha_{3}+\alpha_{4}),\alpha^{2}_{4}\right\rangle_{2}$ equals $\alpha_{4}\left\langle\alpha_{1},\alpha_{2},\alpha_{3},\alpha_{4}\right\rangle_{2}$. The final argument is now that the subgroups $\langle\alpha_{1},\alpha_{2},\alpha_{3},\alpha_{4}\rangle_{2}$ and $\alpha_{4}\left\langle\alpha_{1},\alpha_{2},\alpha_{3},\alpha_{4}\right\rangle_{2}$ should coincide which is in general not the case.
%\end{proof}

In Lemmas \ref{projequiv1} and \ref{projequiv2} we proved that the KM-arcs of type $q/8$ constructed in Theorem \ref{consq8} are $\PGammaL$-equivalent under certain conditions. We will now prove similar results for the KM-arcs of type $q/16$ introduced above.

\begin{lemma}\label{projequiv1bis}
	Let $\{\alpha_{1},\alpha_{2},\alpha_{3},\alpha_{4}\}\subset\F^{*}_{q}$ and $\{\alpha'_{1},\alpha'_{2},\alpha'_{3},\alpha_{4}\}\subset\F^{*}_{q}$ be both $\F_{2}$-independent sets such that $\left\langle\alpha_{1},\alpha_{2},\alpha_{3},\alpha_{4}\right\rangle=\left\langle\alpha'_{1},\alpha'_{2},\alpha'_{3},\alpha_{4}\right\rangle$ and such that $\frac{\alpha^{2}_{i}}{\alpha_{4}}\in\left\langle\alpha_{1},\alpha_{2},\alpha_{3},\alpha_{4}\right\rangle$ for $i=1,2,3$. Let $\mathcal{A}$ be the KM-arc constructed through Theorem \ref{cons} using the tuple $(\alpha_{1},\alpha_{2},\alpha_{3},\alpha_{4})$ and let $\mathcal{A}'$ be the KM-arc constructed through Theorem \ref{cons} using the tuple $(\alpha'_{1},\alpha'_{2},\alpha'_{3},\alpha_{4})$. Then $\mathcal{A}$ and $\mathcal{A}'$ are $\PGammaL$-equivalent.% and $\frac{\alpha'^{2}_{i}}{\alpha_{4}}\in\left\langle\alpha'_{1},\alpha'_{2},\alpha'_{3},\alpha_{4}\right\rangle$
\end{lemma}
\begin{proof}
	We note that it follows from $\frac{\alpha^{2}_{i}}{\alpha_{4}}\in\left\langle\alpha_{1},\alpha_{2},\alpha_{3},\alpha_{4}\right\rangle$ for $i=1,2,3$ and $\left\langle\alpha_{1},\alpha_{2},\alpha_{3},\alpha_{4}\right\rangle=\left\langle\alpha'_{1},\alpha'_{2},\alpha'_{3},\alpha_{4}\right\rangle$ that also $\frac{\alpha'^{2}_{i}}{\alpha_{4}}\in\left\langle\alpha'_{1},\alpha'_{2},\alpha'_{3},\alpha_{4}\right\rangle$ for $i=1,2,3$.
	We proceed as in the proof of Lemma \ref{projequiv1}. We can find a matrix $M\in\GL_{4}(\F_{2})$ such that $(\alpha_{1},\alpha_{2},\alpha_{3},\alpha_{4})M=(\alpha'_{1},\alpha'_{2},\alpha'_{3},\alpha_{4})$. This matrix $M$ is contained in the subgroup $H=\{C\in\GL_{4}(\F_{2})\mid C_{i,4}=0,\: i=1,2,3\}$. This multiplicative group $H$ is generated by the matrices $M_{1}=\left(\begin{smallmatrix}
	1&1&0&0\\0&1&0&0\\0&0&1&0\\1&0&0&1
	\end{smallmatrix}\right)$ and $M_{2}=\left(\begin{smallmatrix}
	0&1&0&0\\0&0&1&0\\1&0&0&0\\0&0&0&1
	\end{smallmatrix}\right)$ and hence it is sufficient to prove the statement for $M=M_{1}$ and for $M=M_{2}$. Let $\alpha,\beta_{1},\beta_{2},\beta_{3}$ be as in the construction presented in Theorem \ref{cons}, so $\Tr(\alpha_{i}\beta_{j})=1$, and $\alpha\in\left\langle\alpha_{1},\alpha_{2},\alpha_{3},\alpha_{4}\right\rangle$ such that $\{\alpha_{1}(\alpha_{1}+\alpha_{4}),\alpha_{2}(\alpha_{2}+\alpha_{4}),\alpha_{3}(\alpha_{3}+\alpha_{4}),\alpha_{4}\alpha\}$ is an $\F_{2}$-independent set.
	\par We first look at $M=M_{1}$. In this case $(\alpha'_{1},\alpha'_{2},\alpha'_{3})=(\alpha_{1}+\alpha_{4},\alpha_{1}+\alpha_{2},\alpha_{3})$. Then $(\beta'_{1},\beta'_{2},\beta'_{3})=(\beta_{1}+\beta_{2},\beta_{2},\beta_{3})$ fulfils $\Tr(\alpha'_{i}\beta'_{j})=\delta_{i,j}$. We know that the construction in Theorem \ref{cons} does not depend on the choice of the coset leaders, so when constructing the KM-arc using the tuple $(\alpha'_{1},\alpha'_{2},\alpha'_{3},\alpha_{4})$ we may use $\beta'_{1},\beta'_{2},\beta'_{3}$ as coset leaders. A straightforward calculation shows that
	\[
	\left\langle\alpha_{1}(\alpha_{1}+\alpha_{4}),\alpha_{2}(\alpha_{2}+\alpha_{4}),\alpha_{3}(\alpha_{3}+\alpha_{4})\right\rangle=\left\langle\alpha'_{1}(\alpha'_{1}+\alpha_{4}),\alpha'_{2}(\alpha'_{2}+\alpha_{4}),\alpha'_{3}(\alpha'_{3}+\alpha_{4})\right\rangle\;,
	\]
	hence $\{\alpha'_{1}(\alpha'_{1}+\alpha_{4}),\alpha'_{2}(\alpha'_{2}+\alpha_{4}),\alpha'_{3}(\alpha'_{3}+\alpha_{4}),\alpha_{4}\alpha\}$ is also an $\F_{2}$-independent set. The point set of $\mathcal{A}$ is given by $\mathcal{S}_{0}\cup\bigcup_{v\in\F^{4}_{2}}\mathcal{S}_{v}$ with
	\begin{align*}
	\mathcal{S}_{0}&=\{(0,1,x)\mid \Tr(\alpha_{i}(\alpha_{i}+\alpha_{4})x)=0,\: i=1,2,3\ \wedge\ \Tr(\alpha_{4}\alpha x)=1\}\text{ and }\\
	\mathcal{S}_{(\lambda_{1},\lambda_{2},\lambda_{3},\lambda_{4})}&=\left\{\left(1,\sum^{4}_{i=1}\lambda_{i}\alpha_{i},\sum^{3}_{i=1}f_{i}(\lambda_{1},\lambda_{2},\lambda_{3})\beta_{i}+s\right)\Bigg|\ s\in S\right\}\;,
	\end{align*}
	and the point set of $\mathcal{A}'$ is given by $\mathcal{S}'_{0}\cup\bigcup_{v\in\F^{3}_{2}}\mathcal{S}'_{v}$ with 	
	\begin{align*}
	\mathcal{S}'_{0}&=\{(0,1,x)\mid \Tr(\alpha'_{i}(\alpha'_{i}+\alpha_{4})x)=0,\: i=1,2,3\ \wedge\ \Tr(\alpha_{4}\alpha x)=1\}\text{ and }\\
	\mathcal{S}'_{(\lambda_{1},\lambda_{2},\lambda_{3},\lambda_{4})}&=\left\{\left(1,\sum^{3}_{i=1}\lambda_{i}\alpha'_{i}+\lambda_{4}\alpha_{4},\sum^{3}_{i=1}f_{i}(\lambda_{1},\lambda_{2},\lambda_{3})\beta'_{i}+s\right)\Bigg|\ s\in S\right\}\;.
	\end{align*}
	Note that $S=\{x\mid \Tr(\alpha_i x)=0\}=\{x\mid\Tr(\alpha'_i x)=0\}$.
	\par Since $\left\langle\alpha_{1}(\alpha_{1}+\alpha_{4}),\alpha_{2}(\alpha_{2}+\alpha_{4}),\alpha_{3}(\alpha_{3}+\alpha_{4})\right\rangle=\left\langle\alpha'_{1}(\alpha'_{1}+\alpha_{4}),\alpha'_{2}(\alpha'_{2}+\alpha_{4}),\alpha'_{3}(\alpha'_{3}+\alpha_{4})\right\rangle$, we know that $\mathcal{S}'_{0}=\mathcal{S}_{0}$. Keeping this in mind and using the above expressions for the $\alpha'_{i}$'s and $\beta'_{i}$'s, it can readily be checked that the collineation induced by the matrix $\left(\begin{smallmatrix}
	1&0&0\\\alpha_{1}+\alpha_{3}&1&0\\\beta_{3}&0&1
	\end{smallmatrix}\right)$ maps $\mathcal{A}$ onto $\mathcal{A'}$.
	\par Now we look at the case $M=M_{2}$, hence at $(\alpha'_{1},\alpha'_{2},\alpha'_{3})=(\alpha_{3},\alpha_{1},\alpha_{2})$. Then $(\beta'_{1},\beta'_{2},\beta'_{3})=(\beta_{3},\beta_{1},\beta_{2})$ fulfils $\Tr(\alpha'_{i}\beta'_{j})=\delta_{i,j}$. Again we can construct $\mathcal{A'}$ (note that $\mathcal{A}$ is as above). Here we can check that the collineation induced by the matrix $\left(\begin{smallmatrix}
	1&0&0\\\alpha_{1}+\alpha_{2}&1&0\\\beta_{2}+\beta_{3}&0&1
	\end{smallmatrix}\right)$ maps $\mathcal{A}$ onto $\mathcal{A'}$.
\end{proof}

\begin{lemma}\label{projequiv2bis}
	Let $\mathcal{A}$ be a KM-arc of type $q/16$ in $\PG(2,q)$ obtained by the construction in Theorem \ref{cons} using the admissible tuple $(\alpha_{1},\alpha_{2},\alpha_{3},\alpha_{4})$ and let $\mathcal{A'}$ be a KM-arc of type $q/16$ in $\PG(2,q)$ obtained by the construction in Theorem \ref{cons} using $(k\alpha^{\varphi}_{1},k\alpha^{\varphi}_{2},k\alpha^{\varphi}_{3},k\alpha^{\varphi}_{4})$, with $k\in\F^{*}_{q}$ and $\varphi$ a field automorphism of $\F_{q}$. Then $\mathcal{A}$ and $\mathcal{A'}$ are $\PGammaL$-equivalent. 
\end{lemma}
\begin{proof}
	We note that the condition $\forall i: \frac{\alpha^{2}_{i}}{\alpha_{4}}\in\left\langle\alpha_{1},\alpha_{2},\alpha_{3},\alpha_{4}\right\rangle$ and the condition $\forall i: \frac{(k\alpha^{\varphi}_{i})^{2}}{k\alpha^{\varphi}_{4}}\in\left\langle k\alpha_{1},k\alpha_{2},k\alpha_{3},k\alpha_{4}\right\rangle$ are equivalent. We also note that if $\alpha\in\left\langle\alpha_{1},\alpha_{2},\alpha_{3},\alpha_{4}\right\rangle$ is such that $\{\alpha_{1}(\alpha_{1}+\alpha_{4}),\alpha_{2}(\alpha_{2}+\alpha_{4}),\alpha_{3}(\alpha_{3}+\alpha_{4}),\alpha_{4}\alpha\}$ is an $\F_{2}$-independent set, then $k\alpha^{\varphi}\in\left\langle k\alpha^{\varphi}_{1},k\alpha^{\varphi}_{2},k\alpha^{\varphi}_{3},k\alpha^{\varphi}_{4}\right\rangle$ is such that $\{k\alpha^{\varphi}_{1}(k\alpha^{\varphi}_{1}+k\alpha^{\varphi}_{4}),k\alpha^{\varphi}_{2}(k\alpha^{\varphi}_{2}+k\alpha^{\varphi}_{4}),k\alpha^{\varphi}_{3}(k\alpha^{\varphi}_{3}+k\alpha^{\varphi}_{4}),(k\alpha^{\varphi}_{4})(k\alpha^{\varphi})\}$ is an $\F_{2}$-independent set. The rest of the proof is analogous to the proof of Lemma \ref{projequiv2}.
\end{proof}

We proved before that the KM-arcs of type $q/8$ constructed in Theorem \ref{cons} are not translation KM-arcs. This also true for the KM-arcs of type $q/16$.

\begin{theorem}\label{nottranslation16}
	Any KM-arc in $\PG(2,q)$ constructed through Theorem \ref{consq8} is not a translation KM-arc.
\end{theorem}
\begin{proof}
	The proof is similar to the proof of Theorem \ref{nottranslation}, now using Remark \ref{smallq} and Lemmas \ref{nottrans16} and \ref{projequiv2bis}.
\end{proof}

The following result is the analogue of Theorems \ref{gwremark} and \ref{iteratie}. Its proof is similar to the proof of Theorem \ref{iteratie}. 

\begin{theorem}\label{iteratie16}Let $\mathcal{A}_{0}$ be the KM-arc of type $q/16$ with $q/16$-nucleus $N(0,0,1)$ and elation line $X=0$ constructed from Theorem \ref{consq8} by the admissible tuple $(\alpha_{1},\alpha_{2},\alpha_{3},\alpha_{4})$ and let $\beta$ be a collineation that stabilises $N$. Let $\mathcal{A'}$ be a KM-arc of type $q^{h}/16$ in $\PG(2,q^{h})$ obtained from $\mathcal{A}_0^\beta$ through Construction \ref{gwc} (B) or (C). Then $\mathcal{A'}$ is $\PGammaL$-equivalent to a KM-arc obtained by the construction in Theorem \ref{consq8} using the admissible tuple $(\alpha_{1},\alpha_{2},\alpha_{3},\alpha_{4})$.
\end{theorem}

We now discuss the existence of the family of KM-arcs presented in Theorem \ref{cons}.

\begin{theorem}\label{admissible}
	A KM-arc $\mathcal{A}$ of type $q/16$ in $\PG(2,q)$, $q=2^{h}$, constructed through Theorem \ref{cons} exists if and only if
	\begin{itemize}
		\item $4\mid h$ and $\mathcal{A}$ is $\PGammaL$-equivalent to the KM-arc constructed through Theorem \ref{cons} using an admissible tuple $(\alpha_{1},\alpha_{2},\alpha_{3},1)$ with $\left\langle\alpha_{1},\alpha_{2},\alpha_{3},1\right\rangle=\F_{16}\subset\F_{q}$,
		\item $6\mid h$ and $\mathcal{A}$ is $\PGammaL$-equivalent to the KM-arc constructed through Theorem \ref{cons} using an admissible tuple $(\alpha_{1},\alpha_{2},\alpha_{3},1)$ with $\left\langle\alpha_{1},\alpha_{2},\alpha_{3},1\right\rangle=\left\langle\F_{4},\F_{8}\right\rangle\subseteq\F_{q}$ or
		\item $7\mid h$ and $\mathcal{A}$ is $\PGammaL$-equivalent to the KM-arc constructed through Theorem \ref{cons} using the admissible tuple $(z,z^{2},z^{4},1)$ with $z\in\F_{q}$ admitting $z^{7}=z+1$ or to the KM-arc constructed through Theorem \ref{cons} using the admissible tuple $(z^{11},z^{22},z^{44},1)$ with $z\in\F_{q}$ admitting $z^{7}=z+1$.
	\end{itemize}
    Here we consider the subfields as additive subgroups of $\F_{q},+$.
\end{theorem}
\begin{proof}
	There exists a KM-arc constructed through Theorem \ref{cons} in $\PG(2,q)$ if we can find a tuple $(\alpha_{1},\alpha_{2},\alpha_{3},\alpha_{4})\in\F^{4}_{q}$ such that $\left\langle\alpha_{1},\alpha_{2},\alpha_{3},\alpha_{4}\right\rangle$ has order 16 and such that $\frac{\alpha^{2}_{i}}{\alpha_{4}}\in\left\langle\alpha_{1},\alpha_{2},\alpha_{3},\alpha_{4}\right\rangle$ for $i=1,2,3$. So we look for all admissible tuples $(\alpha_{1},\alpha_{2},\alpha_{3},\alpha_{4})\in\F^{4}_{q}$. By Lemma \ref{projequiv2bis} we can assume that $\alpha_{4}=1$. We denote $T=\left\langle\alpha_{1},\alpha_{2},\alpha_{3},1\right\rangle$ We distinguish between different cases and subcases. In this discussion we denote the trace function $\F_{q'}\to\F_{2}$ by $\Tr_{q'}$.
	\begin{enumerate}
		\item \textit{We assume that $\F_{4}\subset T$.} In this case $2\mid h$. By Lemma \ref{projequiv1bis} we may assume that $\left\langle\alpha_{1},1\right\rangle=\F_{4}$. It follows that $\frac{\alpha^{2}_{1}}{\alpha_{4}}=\alpha^{2}_{1}\in\left\langle\alpha_{1},1\right\rangle\subset T$. It is immediate that for any $x\in\F_{q}\setminus\F_{4}$ also $x^{2}\in\F_{q}\setminus\F_{4}$. Further, if $x^{2}+x\in\F_{4}$ for an element $x\in\F_{q}\setminus\F_{4}$, then $x^{2}+x\in\F_{4}\setminus\F_{2}$. 
		\begin{enumerate}
			\item \textit{There is an element $x\in T\setminus\F_{4}$ such that $x^{2}\in\left\langle\F_{4},x\right\rangle$.} By Lemma \ref{projequiv1bis} we can put $\alpha_{2}=x$. By the arguments above we know that $\alpha^{2}_{2}\notin\F_{4}$ and that $\alpha^{2}_{2}+\alpha_{2}\in\F_{4}\setminus\F_{2}$, hence $(\alpha^{2}_{2}+\alpha_{2})^{2}+\alpha^{2}_{2}+\alpha_{2}=\alpha^{4}_{2}+\alpha_{2}=1$. So, $\F_{16}\subset\F_{q}$ equivalently $4\mid h$, and $\left\langle\alpha_{1},\alpha_{2},1\right\rangle=\left\langle\alpha_{2},\alpha^{2}_{2},1\right\rangle$ equals $\{x\in\F_{16}\subset\F_{q}\mid\Tr_{16}(x)=0\}$.
			\par The element $\alpha_{3}\in T\setminus\left\langle\alpha_{1},\alpha_{2},1\right\rangle$ must fulfil $\alpha^2_{3}\in T$, hence $\alpha^2_{3}\in\left\langle\alpha_{1},\alpha_{2},1\right\rangle$ or $\alpha^2_{3}+\alpha_{3}\in\left\langle\alpha_{1},\alpha_{2},1\right\rangle$. Both conditions imply that $\alpha_{3}\in\F_{16}$. Consequently, $T=\F_{16}\subset\F_{q}$.
			\item \textit{For any element $x\in T\setminus\F_{4}$ we have $x^{2}\notin\left\langle\F_{4},x\right\rangle$.} We know that $\alpha^{2}_{2}\in T$ and since it is not contained in $\left\langle\alpha_{1},\alpha_{2},1\right\rangle$, we can write $T=\left\langle\alpha_{1},\alpha_{2},\alpha^{2}_{2},1\right\rangle$. Now we must have $\alpha^{4}_{2}\in T$. So we can find $a,b\in\F_{2}$ and $t\in\F_{4}$ such that $\alpha^{4}_{2}+a\alpha^{2}_{2}+b\alpha_{2}+t=0$. If $b=0$, then $\alpha^{2}_{2}+a\alpha_{2}+t^{2}=0$, contradicting the assumption. If $(a,b)=(0,1)$, then $0=(\alpha^{4}_{2}+\alpha_{2}+t)^{4}+\alpha^{4}_{2}+\alpha_{2}+t=\alpha^{16}_{2}+\alpha_{2}$ and hence $T=\F_{16}$, contradicting the assumption.
			\par So, we may assume that $\alpha^{4}_{2}+\alpha^{2}_{2}+\alpha_{2}+t=0$ for some $t\in\F_{4}$. It follows that $(\alpha_{2}+t^{2})^{4}+(\alpha_{2}+t^{2})^{2}+(\alpha_{2}+t^{2})=0$, hence that
			\begin{align*}
			0&=\left[(\alpha_{2}+t^{2})^{4}+(\alpha_{2}+t^{2})^{2}+(\alpha_{2}+t^{2})\right]^{2}+(\alpha_{2}+t^{2})^{4}+(\alpha_{2}+t^{2})^{2}+(\alpha_{2}+t^{2})\\&=(\alpha_{2}+t^{2})^{8}+(\alpha_{2}+t^{2})\;.
			\end{align*}
			The element $\alpha_{2}+t^{2}\in T\setminus\F_{2}$ is thus a generator of $\F_{8}$. Consequently, on the one hand $3\mid h$ and on the other hand $\F_{8}=\left\langle\alpha_{2}+t^{2},(\alpha_{2}+t^{2})^{2},1\right\rangle\subset T$. So, since also $2\mid h$, we have $6\mid h$ and since also $\F_{4}\subset T$, we have that $T=\left\langle\F_{4},\F_{8}\right\rangle$.
		\end{enumerate}
		\item \textit{We assume that $\F_{4}\not\subset T$.} In this case $x^{2}\notin\left\langle x,1\right\rangle$ for any $x\in T\setminus\{0,1\}$, but by the assumption on $T$ we have $x^{2}\in T$. Since also $x^{4}\notin\left\langle x^{2},1\right\rangle$, we have two possibilities. % so we may assume $\alpha_{2}=\alpha^{2}_{1}$ without loss of generality.
		\begin{enumerate}
			\item \textit{There is an element $x\in T\setminus\{0,1\}$ such that $x^{4}+x\in\left\langle x^{2},1\right\rangle$.} By Lemma \ref{projequiv2} we may assume without loss of generality that $\alpha^{4}_{1}+\alpha_{1}\in\left\langle\alpha^{2}_{1},1\right\rangle$ and that $\alpha_{2}=\alpha^{2}_{1}$.
			\par If $\alpha^{4}_{1}+\alpha_{1}=0$ or $\alpha^{4}_{1}+\alpha_{1}=1$, then we have $\left\langle\alpha_{1},1\right\rangle=\F_{4}$ or $\left\langle\alpha^{2}_{1}+\alpha_{1},1\right\rangle=\F_{4}$, respectively, a contradiction. So, $\alpha^{4}_{1}+\alpha_{1}=\alpha^{2}_{1}$ or $\alpha^{4}_{1}+\alpha_{1}=\alpha^{2}_{1}+1$. In both cases we find that $\left\langle\alpha_{1},\alpha^{2}_{1},1\right\rangle=\F_{8}$ and hence that $3\mid h$. The element $\alpha_{3}\in T\setminus\F_{8}$ either fulfils $\alpha^{2}_{3}\in\F_{8}$ or $\alpha^{2}_{3}+\alpha_{3}\in\F_{8}$. The former would imply that $\alpha_{3}\in\F_{8}$, a contradiction, so $\alpha^{2}_{3}+\alpha_{3}+t=0$ for some $t\in\F_{8}$, where $\Tr_{8}(t)=1$ since otherwise $\alpha_3^2+\alpha_3+t=0$ would have a solution $\alpha_3$ in $\F_8$. Furthermore, $\Tr_{q}(t)=0$. Since for $t\in \F_8$, $\Tr_{q}(t)=\frac{h}{3}\Tr_{8}(t)$, we know that $6\mid h$.
			\par Moreover, from $\alpha^{2}_{3}+\alpha_{3}=t$ it follows that $(\alpha_{3}+\alpha_{1})^{2}+(\alpha_{3}+\alpha_{1})=t+\alpha^{2}_{1}+\alpha_{1}$, that $(\alpha_{3}+\alpha^{2}_{1})^{2}+(\alpha_{3}+\alpha^{2}_{1})=t+\alpha^{4}_{1}+\alpha^{2}_{1}$ and that $(\alpha_{3}+\alpha^{2}_{1}+\alpha_{1})^{2}+(\alpha_{3}+\alpha^{2}_{1}+\alpha_{1})=t+\alpha^{4}_{1}+\alpha_{1}$. The elements $t,t+\alpha^{2}_{1}+\alpha_{1},t+\alpha^{4}_{1}+\alpha^{2}_{1},t+\alpha^{4}_{1}+\alpha_{1}$ are the elements of the set $\{x\in\F_{8}\mid\Tr_{8}(x)=1\}$, among which is $1$. By Lemma \ref{projequiv1bis} we can replace $\alpha_{3}$ by $\alpha_{3}+\alpha_{1}$, $\alpha_{3}+\alpha^{2}_{1}$ or $\alpha_{3}+\alpha^{2}_{1}+\alpha_{1}$, so without loss of generality we may assume $t=1$. However, now it immediately follows that $\left\langle\alpha_{3},1\right\rangle=\F_{4}$, a contradiction.
			\item \textit{For any element $x\in T\setminus\{0,1\}$ we have $x^{4}\notin\left\langle x,x^{2},1\right\rangle$}. In particular we have that $\alpha^{4}_{1}\notin\left\langle\alpha_{1},\alpha^{2}_{1},1\right\rangle$. In this case clearly $T=\left\langle\alpha_{1},\alpha^{2}_{1},\alpha^{4}_{1},1\right\rangle$ and also $\alpha^{8}_{1}\in T$. Hence, there exist $a,b,c,d\in\F_{2}$ such that $\alpha^{8}_{1}=a\alpha^{4}_{1}+b\alpha^{2}_{1}+c\alpha_{1}+d$. If $c=0$, then $\alpha^{4}_{1}=a\alpha^{2}_{1}+b\alpha_{1}+d$ contradicting the assumption. We now look at all cases with $c=1$.
			\par If $\alpha^{8}_{1}=\alpha_{1}$, then $\alpha_{1}$ generates the subfield $\F_{8}$, and hence  $\alpha^{4}_{1}\in\left\langle\alpha_{1},\alpha^{2}_{1},1\right\rangle$, a contradiction. If $\alpha^{8}_{1}=\alpha_{1}+1$, then $\{x\in\F_{q}\mid x^{8}+x+1=0\}=\left\langle\alpha_{1},\alpha^{2}_{1},\alpha^{4}_{1}\right\rangle\subset T$. Clearly, this set is a coset of $\F_{8}$ (considered as an additive subgroup of $\F_{q}$). So, $T=\left\langle\alpha_{1},\alpha^{2}_{1},\alpha^{4}_{1},1\right\rangle$ contains $\F_{8}$ and we can find an element in $T$ contradicting the assumption.
			\par If $\alpha^{8}_{1}=\alpha^{2}_{1}+\alpha_{1}$, then $\alpha^{7}_{1}=\alpha_{1}+1$. The element $\alpha_{1}$ thus generates a subfield $\F_{128}$, and so $7\mid h$. We find the first example from the third bullet point in the statement of the theorem. If $\alpha^{8}_{1}=\alpha^{2}_{1}+\alpha_{1}+1$, then $(\alpha_{1}+1)^{8}=(\alpha_{1}+1)^{2}+(\alpha_{1}+1)=0$, hence the element $\alpha_{1}+1\in T$ generates a subfield $\F_{128}$. By Lemma \ref{projequiv1bis} we find the same conclusion as in the previous case.
			\par If $\alpha^{8}_{1}=\alpha^{4}_{1}+\alpha_{1}$, then $\alpha^{7}_{1}=\alpha^{3}_{1}+1$. The element $\alpha_{1}$ thus generates a subfield $\F_{128}$, and so $7\mid h$. If $z\in\F_{128}$ is an element admitting $z^{7}=z+1$, then one can check that $\alpha_{1}=z^{11\cdot2^{k}}$ for some $k=0,\dots,6$. Using Lemmas \ref{projequiv1bis} and \ref{projequiv2bis} we find the second example from the third bullet point in the statement of the theorem. If $\alpha^{8}_{1}=\alpha^{4}_{1}+\alpha_{1}+1$, then $(\alpha_{1}+1)^{8}=(\alpha_{1}+1)^{4}+(\alpha_{1}+1)$, hence the element $\alpha_{1}+1\in T$ generates a subfield $\F_{128}$. By Lemma \ref{projequiv1bis} we find the same conclusion as in the previous case.
			\par If $\alpha^{8}_{1}=\alpha^{4}_{1}+\alpha^{2}_{1}+\alpha_{1}$, then $(\alpha^{2}_{1}+\alpha_{1})^{4}=\alpha^{2}_{1}+\alpha_{1}$, hence $\F_{4}=\left\langle1,\alpha^{2}_{1}+\alpha_{1}\right\rangle\subset T$, a contradiction. If $\alpha^{8}_{1}=\alpha^{4}_{1}+\alpha^{2}_{1}+\alpha_{1}+1$, then $(\alpha^{4}_{1}+\alpha_{1})^{2}=(\alpha^{4}_{1}+\alpha_{1})+1$, hence $\F_{4}=\left\langle1,\alpha^{4}_{1}+\alpha_{1}\right\rangle\subset T$, a contradiction.
	    \end{enumerate}	
	\end{enumerate}
    We conclude that in all cases we find an admissible tuple corresponding to one of the KM-arcs of type $q/16$ given in the statement of the theorem. It also follows from the proof that these tuples are indeed admissible.
\end{proof}

\begin{remark}\label{nieuw} In the case where $7\mid h$, we have found in the previous theorem that $\mathcal{A}$ is $\PGammaL$-equivalent to the KM-arc constructed through Theorem \ref{cons} using the admissible tuple $(z,z^{2},z^{4},1)$ or to the KM-arc constructed through Theorem \ref{cons} using the admissible tuple $(z^{11},z^{22},z^{44},1)$ with $z\in\F_{128}$ admitting $z^{7}=z+1$. We now check that these two possible KM-arcs, say $\mathcal{A}_1$ constructed using $(z,z^{2},z^{4},1)$ and $\mathcal{A}_2$ constructed using $(z^{11},z^{22},z^{44},1)$, are not $\PGammaL$-equivalent.
\par Let  $S_1$ be the subgroup used in the construction of $\mathcal{A}_1$, i.e. the set $\{x\in \F_{q}\mid \Tr(z^{i}x)=0, i=0,1,2,4\}$, and let $S_2$ be the subgroup used in the construction of $\mathcal{A}_2$, i.e. the set $\{x\in \F_{q}\mid \Tr(z^{i}x)=0, i=0,11,22,44\}$. In order for $\mathcal{A}_1$ and $\mathcal{A}_2$ to be $\PGammaL$-equivalent, there has to exist a collineation mapping the points of $\mathcal{A}_1$ on a $q/16$-secant of $\mathcal{A}_1$ onto the points of $\mathcal{A}_2$ on a $q/16$-secant of $\mathcal{A}_2$. This in turn implies that there has to be an automorphism $\phi\in \Aut(\F_{q})$ and an element $k\in\F_{q}$ such that $kS_{1}^{\phi}=S_{2}$. Hence, for such couple $(\phi,k)$ we have
\begin{align*}
  \left\langle z^{11},z^{22},z^{44},1\right\rangle&=\{x\in\F_{q}\mid \forall s\in S_{2}:\Tr(sx)=0\}=\{x\in\F_{q}\mid \forall s\in kS^{\phi}_{1}:\Tr(sx)=0\}\\
  &=\{x\in\F_{q}\mid \forall s\in S_{1}:\Tr(ks^{\phi}x)=0\}=\{k^{-1}y^{\phi}\in\F_{q}\mid \forall s\in S_{1}:\Tr(s^{\phi}y^{\phi})=0\}\\
  &=\{k^{-1}y^{\phi}\in\F_{q}\mid \forall s\in S_{1}:\Tr(sy)=0\}=\{k^{-1}y^{\phi}\in\F_{q}\mid y\in\left\langle z,z^{2},z^{4},1\right\rangle\}\\
  &=k^{-1}\left\langle z,z^{2},z^{4},1\right\rangle^{\phi}
\end{align*}
Both $\left\langle z^{11},z^{22},z^{44},1\right\rangle$ and $\left\langle z^{11},z^{22},z^{44},1\right\rangle$ are contained in $\F_{128}\subseteq\F_{q}$. Any automorphism of $\F_{q}$ fixes the subfield $\F_{128}$, hence $k\in\F_{128}$ and for the restriction $\phi'\in \Aut(\F_{128})$ of $\phi$ to $\F_{128}$, we have $\left\langle z^{11},z^{22},z^{44},1\right\rangle=k^{-1}\left\langle z,z^{2},z^{4},1\right\rangle^{\phi'}$. One can check by computer that there is no couple $(\phi',k)\in\Aut(\F_{128})\times\F_{128}$, hence no couple $(\phi,k)$.
\end{remark}

The above theorem makes clear that the construction from Theorem \ref{cons} can only be applied for specific values of $q$. We now look at some small values of $q$.

\begin{remark}\label{smallq}
	The construction in Theorem \ref{cons} requires $q>32$, but it can be seen that applying this construction for $q=32$ would yield an elation hyperoval. However, it follows immediately from Theorem \ref{admissible} that there exists no admissible tuple in $\F^{4}_{32}$, hence we cannot apply the construction in Theorem \ref{cons}.
	\par By Theorem \ref{admissible} we can, up to $\PGammaL$-equivalence, construct a unique KM-arc of type $4$  in $\PG(2,64)$ through the construction in Theorem \ref{cons}. In \cite{vdd2} already a KM-arc of type $4$ in $\PG(2,64)$ was described. It can be checked that this KM-arc is an elation KM-arc. Moreover, the KM-arc described in \cite{vdd2} is $\PGammaL$-equivalent to the KM-arc that can be constructed through Theorem \ref{cons} using the subgroup $\left\langle\F_{4},\F_{8}\right\rangle$. The automorphism group $G_{64}$ of this KM-arc of type $4$ in $\PG(2,64)$ has size 192. Its subgroup $H_{64}=G_{64}\cap\PGL(3,64)$ of collineations with the identity mapping as field automorphism (subgroup of projectivities) has size $32$ and its subgroup $E_{64}$ of elations has size 8. These 8 elations are the ones described by Theorem \ref{elationsubgroup}. Both the group $G_{64}$ and the group $H_{64}$ have three orbits on the points of the KM-arcs: one orbit containing the four points on the line at infinity, one orbit containing the 32 points of the KM-arc on the $4$-secants with equation $Y=aX$ with $a\in\F_{8}$ and one orbit containing the 32 points of the KM-arc on the $4$-secants with equation $Y=aX$ with $a\notin\F_{8}$.
	\par By Theorem \ref{admissible} and Remark \ref{nieuw}, we know that, up to $\PGammaL$-equivalence, we can construct two different KM-arcs of type $8$  in $\PG(2,128)$ through the construction in Theorem \ref{cons}, say $\mathcal{A}_{128,1}$ and $\mathcal{A}_{128,2}$. Denote the automorphism group of $\mathcal{A}_{128,i}$ by $G_{128,i}$ and denote $H_{128,i}=G_{128,i}\cap\PGL(3,128)$, $i=1,2$. The automorphism groups $G_{128,1}$ and $G_{128,2}$ have order $896$, and their subgroups $H_{128,1}$ and $H_{128,2}$ have order 128. Both the group $G_{128,i}$ and its subgroup $H_{128,i}$ act transitively on the set of affine points of $\mathcal{A}_{128,i}$, $i=1,2$. It should be noted that these KM-arcs are not translation KM-arcs, since the respective groups of elations stabilising the KM-arcs only have size 16 (they equal the subgroup described in Theorem \ref{elationsubgroup}).
\end{remark}

\begin{theorem}\label{LSen2A}
	If $\mathcal{A}$ is a KM-arc of type $q/16$ in $\PG(2,q)$, $q=2^{h}$ and $h\mid 4$, obtained by the construction in Theorem \ref{cons} using an admissible tuple $(\alpha_{1},\alpha_{2},\alpha_{3},1)$ with $\left\langle\alpha_{1},\alpha_{2},\alpha_{3},1\right\rangle=\F_{16}$, then $\mathcal{A}$ also arises by applying Construction \ref{gwc} (A) in $\PG(2,16)$ on the Lunelli-Sce hyperoval.
\end{theorem}
\begin{proof}
	We denote the trace function $\F_{q}\to\F_{2}$ by $\Tr_{q}$, the trace function  $\F_{16}\to\F_{2}$ by $\Tr_{16}$ and the trace function $\F_{q}\to\F_{16}$ by $\Tr_{q,16}$.
	Let $\zeta$ be a generator of $\F_{16}$ admitting $\zeta^{4}=\zeta+1$. By Lemma \ref{projequiv1bis} we may assume that $\alpha_{i}=\zeta^{i}$, $i=1,2,3$. Then, $(\beta_{1},\beta_{2},\beta_{3})=(\zeta^{2},\zeta,1)$ admits $\Tr_{16}(\zeta^{i}\beta_{j})=\delta_{i,j}$, $i,j\in\{1,2,3\}$ and $\Tr_{16}(1\cdot\beta_{j})=0$. Let $\mathcal{S}$ be the set $\{x\in \F_q\mid\forall t\in\F_{16}:\Tr_{q}(tx)=0\}$.
	\par Let $\alpha\in\F_{16}$ be such that $\{\zeta(\zeta+1),\zeta^{2}(\zeta^{2}+1),\zeta^{3}(\zeta^{3}+1),\alpha\}$ is an $\F_{2}$-independent set. Considering $\F_{16}$ as a subfield of $\F_{q}$, we may assume that $\mathcal{A}$ is given by $\mathcal{A}=\mathcal{S}_{0}\cup\bigcup_{v\in\F^{4}_{2}}\mathcal{S}_{v}$ with
	\begin{align*}
	\mathcal{S}_{0}&=\{(0,1,x)\mid \Tr_{q}(\zeta^{i}(\zeta^{i}+1)x)=0,\: i=1,2,3\ \wedge\ \Tr_{q}(\alpha x)=1\}\\
	&=\{(0,1,x)\mid \Tr_{q}(\zeta^{i}x)=0,\: i=0,1,2\ \wedge\ \Tr_{q}(\zeta^{3} x)=1\}	
	\end{align*}
	and 
	\begin{align*}
	\mathcal{S}_{\overline{\lambda}}&=\left\{\left(1,\sum^{3}_{i=1}\lambda_{i}\alpha_{i}+\lambda_{4},\sum^{3}_{i=1}f_{i}(\lambda_{1},\lambda_{2},\lambda_{3})\beta_{i}+s\right)\Bigg|\ s\in S\right\}\\
	&=\left\{\left(1,\sum^{3}_{i=1}\lambda_{i}\zeta^{i}+\lambda_{4},\sum^{3}_{i=1}f_{i}(\lambda_{1},\lambda_{2},\lambda_{3})\zeta^{3-i}+s\right)\Bigg|\ s\in S\right\}
	\end{align*}
	for any $\overline{\lambda}=(\lambda_{1},\dots,\lambda_{4})\in\F^{4}_{2}$.	$\mathcal{A}$ is a KM-arc of type $q/16$ in $\PG(2,q)$ with $q/16$-nucleus $(0,0,1)$.
	\par We define the point set $\mathcal{H}$ in $\PG(2,16)$ by $\mathcal{H}=\mathcal{H}'\cup\{(0,1,1),(0,0,1)\}$ and
	\[
	\mathcal{H}'=\left\{\left(1,\sum^{3}_{i=1}\lambda_{i}\zeta^{i}+\lambda_{4},\sum^{3}_{i=1}f_{i}(\lambda_{1},\lambda_{2},\lambda_{3})\zeta^{3-i}\right) \Bigg| \ (\lambda_{1},\lambda_{2},\lambda_{3},\lambda_{4})\in\F^{4}_{2}\right\}\;.
	\]
	It can readily be checked that $\mathcal{H}$ is the Lunelli-Sce hyperoval, having an automorphism group of size $144$.
	\par Let $k\in\F_{q}$ be such that $\Tr_{q,16}(k)=1$; such an element can always be found. Let $I$ be the additive subgroup of $\F_{q}$ given by $\{x\mid\Tr_{q,16}(kx)=0\}$; it has size $q/16$. If $x\in\F_{16}\subseteq\F_{q}$ admits $\Tr_{q,16}(kx)=0$ then $x=0$. So, $I$ is a direct complement of $\F_{16}$ in $\F_{q}$. Moreover, for any element $x\in\F_{q}$ with $\Tr_{q,16}(kx)=0$ we know that $\Tr_{q}(ktx)=\Tr_{16}(\Tr_{q,16}(kx)t)=0$, for all $t\in\F_{16}\subseteq\F_{q}$. Hence $I$ equals the set $\{x\mid\forall t\in\F_{16}:\Tr_{q}(kxt)=0\}$.
	\par We now apply Construction \ref{gwc} (A) using $\mathcal{H}$ and  $I$ to construct the KM-arc $\mathcal{A'}$ of type $q/16$ in $\PG(2,q)$. The point set of $\mathcal{A'}$ is given by $\mathcal{S}'_{0}\cup\bigcup_{v\in\F^{4}_{2}}\mathcal{S}'_{v}$ with $\mathcal{S}'_{0}=\{(0,1,1+i)\mid i\in I\}$ and
	\begin{align*}
	\mathcal{S}'_{(\lambda_{1},\lambda_{2},\lambda_{3},\lambda_{4})}&=\left\{\left(1,\sum^{3}_{i=1}\lambda_{i}\zeta^{i}+\lambda_{4},\sum^{3}_{i=1}f_{i}(\lambda_{1},\lambda_{2},\lambda_{3})\zeta^{3-i}+i\right)\Bigg|\ i\in I\right\}
	\end{align*}
	for any $(\lambda_{1},\lambda_{2},\lambda_{3},\lambda_{4})\in\F^{4}_{2}$.
	We define the KM-arc $\mathcal{A}''$ in $\PG(2,q)$ using the tuple $(k\zeta,k\zeta^{2},k\zeta^{3},k)$. Its point set is given by $\mathcal{S}''_{0}\cup\bigcup_{v\in\F^{4}_{2}}\mathcal{S}''_{v}$ with
	\[
	\mathcal{S}''_{0}=\{(0,1,x)\mid \Tr_{q}(k^{2}\zeta^{i}x)=0,\: i=0,1,2\ \wedge\ \Tr_{q}(k^{2}\zeta^{3} x)=1\}
	\]
	and
	\[
	\mathcal{S}''_{(\lambda_{1},\lambda_{2},\lambda_{3},\lambda_{4})}=\left\{\left(1,k\sum^{3}_{i=1}\lambda_{i}\zeta^{i}+k\lambda_{4},\sum^{3}_{i=1}f_{i}(\lambda_{1},\lambda_{2},\lambda_{3})\zeta^{3-i}+i\right)\Bigg|\ i\in I\right\}
	\]
	since $\{x\mid\Tr_{q}(k\zeta^{i}x)=0, i=0,1,2,3\}=\{x\mid\forall t\in\F_{16}:\Tr_{q}(ktx)=0\}$ and  $\Tr_{q}((k\zeta^{i})\beta_{j})=\Tr_{16}(\zeta^{i}\zeta^{3-j})=\delta_{i,j}$ for $i=0,1,2,3$ and $j=1,2,3$. Let $\gamma$ be the collineation induced by the trivial field automorphism and the matrix $C'=C\left(\begin{smallmatrix}
	1&0&0\\0&k&0\\0&0&1 
	\end{smallmatrix}\right)$ where we interpret $C$ over $\F_{q}$. It is immediate that $\left(\mathcal{S}'_{\overline{\lambda}}\right)^{\gamma}=\mathcal{S}''_{\overline{\lambda}}$ for all $\overline{\lambda}\in\F^{4}_{2}$. Furthermore,
	\begin{align*}
	\left(\mathcal{S}'_{0}\right)^{\gamma}&=\{(0,k,1+i)\mid i\in I\}=\{(0,1,k^{-1}(1+i))\mid \forall t\in\F_{16}:\Tr_{q}(kit)=0\}\\
	&=\{(0,1,x)\mid\forall t\in\F_{16}: \Tr_{q}(k^{2}xt+kt)=0\}\\
	&=\{(0,1,x)\mid\forall t\in\F_{16}: \Tr_{q}(k^{2}xt)=\Tr_{16}(t)\}\\
	&=\{(0,1,x)\mid\Tr_{q}(k^{2}\zeta^{i}x)=0,\: i=0,1,2\ \wedge\ \Tr_{q}(k^{2}\zeta^{3} x)=1\}\\
	&=\mathcal{S}''_{0}
	\end{align*}
	In the penultimate step we used that $\Tr_{q}$ is $\F_{2}$-linear, that $\left\langle\zeta,\zeta^{2},\zeta^{3},1\right\rangle=\F_{16}$ and that $\Tr_{16}(1)=\Tr_{16}(\zeta)=\Tr_{16}(\zeta^{2})=0$ and $\Tr_{16}(\zeta^{3})=1$.
	\par We conclude that $\mathcal{A'}=(\mathcal{A}'')^{\gamma}$. We also know that $\mathcal{A}$ and $\mathcal{A}''$ are isomorphic since the tuples $(\alpha_{1},\alpha_{2},\alpha_{3},1)$ and $(k\alpha_{1},k\alpha_{2},k\alpha_{3},k)$ give rise to $\PGammaL$-equivalent KM-arcs by Lemma \ref{projequiv2}. This proves the theorem.
\end{proof}

\begin{remark}
	The previous theorem also shows that when applying Theorem \ref{cons} for $q=16$ we get a set of 17 points which forms a hyperoval together with $(0,0,1)$. This hyperoval is the Lunelli-Sce hyperoval.
\end{remark}

We end this section with a discussion on the existence of KM-arcs of type $q/16$.

\begin{remark}\label{q16exist}
	Previously KM-arcs of type $q/16$ in $\PG(2,q)$, $q=2^{h}$ were known to exist for $4\mid h$, $5\mid h$ and $6\mid h$ through Constructions \ref{kmc} and \ref{gwc} (A),  Construction \ref{gwc} (B) and  Construction \ref{gwc} (C) applied on the example of a KM-arc of type $4$ in $\PG(2,64)$ \cite{vdd2}, respectively.
	\par By Theorem \ref{admissible} the construction from Theorem \ref{cons} can only be applied for $4\mid h$, $6\mid h$ and $7\mid h$ and given admissible tuples. The KM-arcs of type $2^{h-4}$ in $\PG(2,2^{h})$ with $4\mid h$, constructed through Theorem \ref{cons} using an admissible tuple $(\alpha_{1},\alpha_{2},\alpha_{3},1)$ with $\left\langle\alpha_{1},\alpha_{2},\alpha_{3},1\right\rangle=\F_{16}\subset\F_{q}$ were already known to exist since they are by Theorem \ref{LSen2A} $\PGammaL$-equivalent to the KM-arcs of type $q/16$ obtained by applying Construction \ref{gwc} (A) on a Lunelli-Sce hyperoval. Note that Construction \ref{gwc} (A) can also be applied on a regular hyperoval. In this case we find a translation KM-arc of type $q/16$, which cannot arise from the construction in Theorem \ref{cons} by Theorem \ref{nottranslation16}.
	\par The KM-arcs of type $2^{h-4}$ in $\PG(2,2^{h})$ with $6\mid h$, constructed through Theorem \ref{cons} using an admissible tuple $(\alpha_{1},\alpha_{2},\alpha_{3},1)$ with $\left\langle\alpha_{1},\alpha_{2},\alpha_{3},1\right\rangle=\left\langle\F_{4},\F_{8}\right\rangle\subset\F_{q}$ were already known to exist since they are by Theorem \ref{iteratie16} and Remark \ref{smallq} $\PGammaL$-equivalent to the KM-arcs of type $q/16$ obtained by applying Construction \ref{gwc} (C) on the KM-arc of type 4 in $\PG(2,64)$ described in \cite{vdd2}. Note that no other KM-arcs of type $2^{h-4}$ in $\PG(2,2^{h})$ with $6\mid h$, are known (unless $h$ is also a multiple of $4$, $5$ or $7$).
	\par The KM-arcs of type $2^{h-4}$ in $\PG(2,2^{h})$ with $7\mid h$, constructed through Theorem \ref{cons} using an admissible tuple $(\alpha_{1},\alpha_{2},\alpha_{3},1)$ with $\left\langle\alpha_{1},\alpha_{2},\alpha_{3},1\right\rangle=\left\langle z,z^{2},z^{4},1\right\rangle\subset\F_{q}$ or $\left\langle\alpha_{1},\alpha_{2},\alpha_{3},1\right\rangle=\left\langle z^{11},z^{22},z^{44},1\right\rangle\subset\F_{q}$ were not described before, so are two new families of examples (the KM-arcs of both families are inequivalent by Remark \ref{nieuw}).
\end{remark}

\end{document}